\documentclass[12pt]{article}
\usepackage{enumerate}
\usepackage[hyperindex,breaklinks,colorlinks,citecolor=blue]{hyperref}
\usepackage[latin1]{inputenc}
\usepackage{amsmath, amsthm, amsfonts, amssymb}
\usepackage{mathrsfs}
\usepackage{latexsym}
\usepackage{fullpage}
\usepackage[all]{xy}
\usepackage{color}
\usepackage{hhline,colortbl}

\newtheorem{lem}{Lemma}[section]
\newtheorem{prop}{Proposition}[section]
\newtheorem{thm}{Theorem}[section]
\newtheorem{cor}{Corollary}[section]

\theoremstyle{defn}
\newtheorem{defn}{Definition}[section]

\theoremstyle{rem}
\newtheorem{rem}{Remark}[section]
\newtheorem{rmks}{Remarks}[section]

\newtheorem*{Examples}{Examples}

\newcommand{\N}{\mathbb{N}}
\newcommand{\Q}{\mathbb{Q}}
\newcommand{\R}{\mathbb{R}}
\newcommand{\cs}{2^\omega}
\newcommand{\fs}{2^{<\omega}}
\newcommand{\K}{\mathrm{K}}

\newcommand{\uh}{\upharpoonright}
\newcommand{\emptystring}{\epsilon}
\newcommand{\Low}[1]{\mathrm{Low}(#1)}
\newcommand{\Lowpair}[2]{\mathrm{Low}(#1,#2)}
\newcommand{\plop}[1]{\langle #1 \rangle}
\newcommand{\Highpair}[2]{\mathrm{High}(#1,#2)}
\newcommand{\SMLR}{\mathrm{W2R}}
\newcommand{\MLR}{\mathrm{MLR}}
\newcommand{\CR}{\mathrm{CR}}

\newcommand{\SR}{\mathrm{SR}}

\newcommand{\jump}{\mathbf{0}'}
\newcommand{\dom}{\mathrm{dom}}

\begin{document}

%\title{Randomness and lowness notions via open sets coverings}
\title{Randomness and lowness\\notions via open covers}
\author{Laurent Bienvenu and Joseph S. Miller}
%\thanks[lb]{Most of this work was done while the first author was a von Humboldt postdoctoral fellow at the Institut f\"ur Informatik, Ruprecht-Karls Universit\"at Heidelberg, Germany.}
%\address{LIAFA, CNRS \& Universit\'e de Paris 7, France}
%\ead{laurent.bienvenu@liafa.jussieu.fr}
%
%
%\author{Joseph S.~Miller\thanksref{jm}}
%\thanks[jm]{The second author was supported by the National Science Foundation under grants DMS-0945187 and DMS-0946325, the latter being part of a Focused Research Group in Algorithmic Randomness.}
%\address{University of Wisconsin$-$Madison, USA}
%%\address{Department of Mathematics\\
%%University of Wisconsin\\
%%Madison, WI 53706-1388, USA}
%\ead{jmiller@math.wisc.edu}

\maketitle

\begin{abstract}
One of the main lines of research in algorithmic randomness is that of lowness notions. Given a randomness notion $\mathscr{R}$, we ask for which sequences~$A$ does relativization to $A$ leave $\mathscr{R}$ unchanged (i.e., $\mathscr{R}^A=\mathscr{R}$)? Such sequences are call \emph{low for $\mathscr{R}$}. This question extends to a pair of randomness notions $\mathscr{R}$ and $\mathscr{S}$, where $\mathscr{S}$ is weaker: for which~$A$ is $\mathscr{S}^A$ still weaker than~$\mathscr{R}$? In the last few years, many results have characterized the sequences that are low for randomness by their low computational strength. A few results have also given measure-theoretic characterizations of low sequences. For example, Kjos-Hanssen (following  Ku{\v c}era) proved that $A$ is low for Martin-L\"of randomness if and only if every $A$-c.e.\ open set of measure less than~$1$ can be covered by a c.e.\ open set of measure less than~$1$. 
 
In this paper, we give a series of results showing that a wide variety of lowness notions can be expressed in a similar way, i.e., via the ability to cover open sets of a certain type by open sets of some other type. This provides a unified framework that clarifies the study of lowness for randomness notions, and allows us to give simple proofs of a number of known results. We also use this framework to prove new results, including showing that the classes $\Lowpair{\MLR}{\SR}$ and $\Lowpair{\SMLR}{\SR}$ coincide, answering a question of Nies. Other applications include characterizations of highness notions, a broadly applicable explanation for why low for randomness is the same as low for tests, and a simple proof that $\Lowpair{\SMLR}{\mathscr{S}} = \Lowpair{\MLR}{\mathscr{S}}$, where $\mathscr{S}$ is the class of Martin-L\"of, computable, or Schnorr random sequences.

The final section gives characterizations of lowness notions using summable functions and convergent measure machines instead of open covers. We finish with a simple proof of a result of Nies, that $\Low{\MLR}=\Lowpair{\MLR}{\CR}$.
\end{abstract}

%---------------------------------------------------------
\section{Introduction}\label{sec:intro}
%---------------------------------------------------------

This paper is organized as follows. In the remainder of this section we review notation, introduce the basic notions, including the relevant randomness classes, and survey what is known about lowness for randomness notions. In Section~\ref{sec:open-covers} we consider Ku{\v c}era's result that $X$ is not Martin-L\"of random iff there is a c.e.\ open set $\mathcal{U}$ of measure less than 1 such that $U$ covers all tails of $X$. We prove analogous theorems for computable and Schnorr randomness by placing further restrictions on the c.e.\ open covers.

In Section~\ref{sec:main-lemma} we prove our main technical lemma and show that it applies to Martin-L\"of randomness, computable randomness and Schnorr randomness. Together with the previous section, the main lemma provides a unified framework to study lowness classes in terms of c.e.\ open covers. Section~\ref{sec:applications} gives a number of applications. Kjos-Hanssen~\cite{Kjos-Hanssen2007} (based on the ideas of Ku{\v c}era) showed that $A$ is low for Martin-L\"of randomness if and only if every $A$-c.e.\ open set of measure less than~$1$ can be covered by a c.e.\ open set of measure less than~$1$. In Section~\ref{ssec:apps-lowness}, we show that a wide variety of lowness notions can be expressed in a similar way, i.e., via the ability to cover open sets of a certain type by open sets of another type. Kjos-Hanssen's result actually gives a characterization of LR-reducibility, and in Section~\ref{ssec:apps-reducibilities}, we note that similar characterizations could be given for the weak reducibilities associated with computable and Schnorr randomness. In Section~\ref{ssec:apps-tests} we give a broadly applicable explanation for why lowness for randomness has, in the cases that have been studied, turned out to be the same as lowness for tests. In Section~\ref{ssec:apps-W2R} we show that $\Lowpair{\SMLR}{\mathscr{S}} = \Lowpair{\MLR}{\mathscr{S}}$ for $\mathscr{S}\in \{\MLR,\CR,\SR\}$. Two of these facts were known, but the Schnorr randomness case answers an open question of Nies~\cite[Problem~8.3.16]{Nies2009}. Finally, Section~\ref{ssec:apps-highness} applies our framework to highness notions, focusing on the poorly understood class $\Highpair{\CR}{\MLR}$.

Section~\ref{sec:other-refumulations} departs from the rest of the paper; in it, we reformulate lowness notions using summable functions and convergent measure machines instead of open covers. A final application is given in Section~\ref{ssec:apps-winning}, where we give a straightforward proof that $\Low{\MLR}=\Lowpair{\MLR}{\CR}$ (Nies \cite{Nies2005,Nies2009}).

\subsection{Basic notation}

We work in Cantor space, in other words, the set $\cs$ of infinite binary sequences. We write $\fs$ for the set of finite binary strings and $\emptystring \in \fs$ for the empty string. If~$S$ is a subset of $\fs$, we define
\begin{align*}
S^n &= \{\sigma \in \fs \; : \; \sigma=\sigma_0\sigma_1\sigma_2\ldots\sigma_{n-1} \text{ s.t.\ } (\forall i<n)\; \sigma_i\in S\}\text{, and}\\
S^\omega &= \{A \in \cs \; : \; A=\sigma_0\sigma_1\sigma_2\ldots \text{ s.t.\ } (\forall i)\; \sigma_i \in S\}.
\end{align*}
For a string~$\sigma$, $[\sigma]$ denotes the cylinder generated by $\sigma$, in other words, the set of infinite sequences with prefix $\sigma$. For $U\subseteq\fs$, the open set generated by $U$ is $[U]=\bigcup_{\sigma\in U} [\sigma]$. We denote the Lebesgue measure on $\cs$ by $\mu$ (a.k.a.\ the uniform measure on $\cs$, which can be defined as the unique probability measure on $\cs$ that satisfies $\mu([\sigma])=2^{-|\sigma|}$). If $U$ is a prefix-free subset of $\fs$, the \emph{measure of $U$} is the quantity $\mu(U) = \mu([U]) = \sum_{\sigma \in U} 2^{-|\sigma|}$. Note that $\mu(U^n) = \mu(U)^n$, again assuming that $U\subset\fs$ is prefix-free. We say that an open set (resp.\ prefix-free set of strings) is \emph{bounded} if its measure is smaller than~$1$. A \emph{c.e.\ open set} (or \emph{$\Sigma^0_1$ class}) is an open set generated by a c.e.\ prefix-free set of strings. We say that a c.e.\ open set (resp.\ c.e.\ prefix-free set of strings) is a \emph{Schnorr set} if its measure is computable. 

If $A \in \cs$, we denote by $A \uh n$ the prefix $A$ of size~$n$, i.e., $A \uh n = A(0)A(1) \ldots A(n-1)$. Also, we call a \emph{tail of $A$} any infinite sequence of type $A(k)A(k+1)A(k+2)\ldots$ for $k \in \N$ (in other words, any sequence obtained by removing a finite prefix from~$A$). If $\mathcal{X}$ is a subset of $\cs$ and $\sigma$ a finite string, we set
\[
(\mathcal{X} \mid \sigma) = \{Z \in \cs\;  :  \; \sigma Z \in \mathcal{X}\}. 
\]
Similarly, if $W$ is a subset of $\fs$ we set
\[
(W \mid \sigma) = \{\tau \in \fs\;  :  \; \sigma\tau \in W\}.
\]
Note that this is consistent with the conditional probability notation: $\mu(\mathcal{X} \mid \sigma)$ is just the measure of $\mathcal{X}$ conditioned by $[\sigma]$, i.e., $\mu(\mathcal{X} \cap [\sigma])/\mu(\sigma)$. Note also that if $\mathcal{U}$ is a c.e.\ open set, then so is $(\mathcal{U} \mid \sigma)$ for all~$\sigma$. If moreover the measure of $\mathcal{U}$ is computable, then so is the measure of $(\mathcal{U} \mid \sigma)$ (uniformly in $\mu(\mathcal{U})$ and $\sigma$).

\subsection{Randomness notions}

In general, a \emph{test} is a non-increasing sequence $(\mathcal{U}_n)_{n \in \N}$ of open sets such that $\bigcap_n \mathcal{U}_n$ has measure~$0$. We say that a sequence~$X \in \cs$ \emph{fails} the test $(\mathcal{U}_n)_{n \in \N}$ if $X \in \bigcap_n \mathcal{U}_n$, and that $X$ \emph{passes the test} otherwise. If $\mathcal{X}$ is a subset of $\cs$, we say that a test $(\mathcal{U}_n)_{n \in \N}$ \emph{covers} $\mathcal{X}$ if $\mathcal{X} \subseteq  \bigcap_n \mathcal{U}_n$. We say that a test $(\mathcal{U}_n)_{n \in \N}$ covers another test $(\mathcal{U}'_n)_{n \in \N}$ if $ \bigcap_n \mathcal{U}'_n \subseteq  \bigcap_n \mathcal{U}_n$.

\begin{defn}
A test $(\mathcal{U}_n)_{n \in \N}$ is a Martin-L\"of test if $\mu(\mathcal{U}_n) \leq 2^{-n}$ for all~$n$. It is a Schnorr test if one further has $\mu(\mathcal{U}_n) = 2^{-n}$. We say that $X$ is Martin-L\"of random if it passes all Martin-L\"of tests, and that $X$ is Schnorr random if it passes all Schnorr tests. We denote by $\MLR$ the set of Martin-L\"of random sequences and by $\SR$ the set of Schnorr random sequences. 
\end{defn}

\begin{rem}\label{rem:test-note}
It should be noticed that the quantity $2^{-n}$ in the above definition is arbitrary: we would get the same classes $\MLR$ and $\SR$ if we replaced it by any $f(n)$, with $f$ a computable function that tends to~$0$. Another important fact is that there exists a \emph{universal} Martin-L\"of test, i.e., a Martin-L\"of test such that for any sequence~$X$, $X$ passes that test if and only~$X$ is Martin-L\"of random. There is no such universal test for Schnorr randomness.
\end{rem}

A third important notion of randomness is computable randomness, whose definition involves the concept of martingale. 

A martingale is a function $d: \fs \rightarrow \R^{\geq 0}$ such that for all $\sigma \in \fs$
\[
d(\sigma)=\frac{d(\sigma 0)+d(\sigma 1)}{2}.
\]
It is said to be \emph{normed} if $d(\emptystring)=1$. We say that a martingale \emph{succeeds} on a sequence $X \in \cs$ if $\limsup d(X \uh n)=+\infty$. 

For any martingale, the set of sequences on which it succeeds has measure~$0$. This is a direct consequence of the so-called Ville-Kolmogorov inequality.

\begin{prop}
Let~$d$ be a martingale, $\sigma \in \fs$ and $q >1$ a real number. Let
\[
\mathcal{U}_{d,\sigma,q}=\big\{X \in \cs \; : \; (\exists n > |\sigma|) \; d(X \uh n) \geq q \cdot d(\sigma)\big\}. 
\]
Then $\mu(\mathcal{U}_{d,\sigma,q} \mid \sigma) \leq 1/q$.
\end{prop}

We can now define the notion of computable randomness.

\begin{defn}
We say that $X$ is \emph{computably random} if no computable martingale succeeds on~$X$. We denote by~$\CR$ the set of computably random sequences. 
\end{defn}

In the above definition, by ``computable'' we mean computable as a real-valued function. However, it will be more convenient in this paper to work with \emph{exactly computable martingales}, i.e., martingales that are rational-valued and computable as functions from $\N$ to $\Q^{\geq 0}$. This is made possible by a note of Lutz~\cite{Lutz2004}, where it is proven that for every computable martingale~$d$, there exists an exactly computable martingale~$d'$ and positive real constants $\alpha,\beta$ such that $\alpha d < d' < \beta d$ (in particular, $d$ and $d'$ succeed on the same set of sequences). Furthermore, an index of~$d'$ can be uniformly computed from an index of~$d$. Therefore, we can equivalently define the set $\CR$ as being the set of sequences $X$ such that no (normed) exactly computable martingale succeeds on~$X$. We can also rephrase the definition in terms of test.

\begin{defn}
Let $d$ be a rational-valued normed martingale and $q$ a rational such that $q>1$. We say that $U \subseteq \fs$ is a $(d,q)$-winning set if for some rational~$q>1$ we have $U=\{\sigma\; : \; \sigma~\text{minimal s.t. } d(\sigma) \geq q\}$. We say that $U \subseteq \fs$ is a winning set if it is a $(d,q)$-winning set for some exactly computable normed martingale~$d$ and rational $q>1$.  We also say that a c.e.\ open set~$\mathcal{U}$ is a winning set if $\mathcal{U}=[U]$ where~$U$ is a winning set of strings. 

Given a normed exactly computable martingale~$d$, the test \emph{induced} by~$d$ is the sequence $(\mathcal{U}_n)_{n \in \N}$ where $\mathcal{U}_n$ is the $(d,2^{n})$-winning set. 
\end{defn}

(in the above definition, and in the rest of the paper, we say that a string $\sigma$ is \emph{minimal} for a given property $P$ if $\sigma$ satisfies $P$ and no prefix of $\sigma$ does). 

Now, we immediately see that $X$ is computably random if and only if $X$ passes all tests induced by normed exactly computable martingales. 

\begin{rem}
If $d$ is a normed, exactly computable martingale, any $(d,q)$-winning set is c.e.\ open, and (by the Ville-Kolmogorov inequality) has measure at most~$1/q$. Thus the test induced by a normed, exactly computable martingale is a Martin-L\"of test. 
\end{rem}

The last randomness notion we will discuss in this paper is a very natural generalization of Martin-L\"of randomness. Weak $2$-randomness (sometimes called Kurtz $2$-randomness) allows tests $(\mathcal{U}_n)_{n \in \N}$ with the looser condition that $\mu(\mathcal{U}_n)$ tends to~$0$, possibly at a non-computable rate. 

\begin{defn}
A generalized Martin-L\"of test is a sequence $(\mathcal{U}_n)_{n \in \N}$ of uniformly c.e.\ open sets such that $\lim_n \mu(\mathcal{U}_n)=0$. We say that $X$ is weak $2$-random if $X$ passes all generalized Martin-L\"of tests and denote by $\SMLR$ the set of weak $2$-random sequences. 
\end{defn}

Note that a generalized Martin-L\"of test is nothing more than a measure zero $\Pi^0_2$ class and $X$ is weak $2$-random iff it avoids every such class.

%Demuth randomness was introduced by Demuth \cite{Demuth1988} in the setting of effective analysis. Like weak $2$-randomness, it generalizes Martin-L\"of randomness. 
%
%\begin{defn}
%Let $(\mathcal{U}_e)_{e \in \N}$ be an effective enumeration of all c.e.\ open sets. A Demuth test is a sequence $(\mathcal{U}_{f(n)})_{n \in \N}$ where~$f$ is an $\omega$-c.e.\ function and $\mu(\mathcal{U}_{f(n)})<2^{-n}$ for all~$n$. A sequence $X \in \cs$ is said to be Demuth random if it passes all Demuth tests and we denote by $\DR$ the set of Demuth random sequences. 
%\end{defn}

\subsection{Lowness notions: the state of the art}

In computability theory, an oracle~$A$ is said to be \emph{low} in a certain context if a relativization to~$A$ ``does not help''. For example, $A$ is low for the Turing jump (usually referred to simply as ``low'') if $A'\equiv_T\jump$. Lowness notions have been very important in the recent development of algorithmic randomness. Take a randomness notion~$\mathscr{R}$ that, like all the above, can be defined via tests. One can relativize the notion of test to an oracle~$A$, getting the class~$\mathscr{R}^A$ of sequences that pass all~$A$-tests. Since taking~$A$ as an oracle gives additional computational power, we have $\mathscr{R}^A \subseteq \mathscr{R}$. We say that~$A$ is \emph{low for the randomness notion~$\mathscr{R}$} if, as an oracle, $A$~has so little computational power that~$\mathscr{R}^A = \mathscr{R}$. We denote by $\Low{\mathscr{R}}$ the set of sequences that are low for~$\mathscr{R}$.

Zambella~\cite{Zambella1990} introduced lowness for Martin-L\"of randomness. A beautiful series of results by Nies and others (see~\cite{Nies2009} for a complete exposition) showed that these oracles have remarkable properties. They proved that $A\in\Low{\MLR}$ if and only if $A$ is low for prefix-free Kolmogorov complexity (i.e., $\K^A=\K+O(1)$), and if and only if $A$ is \emph{$K$-trivial}. This latter property states that the initial segments of $A$ have minimal prefix-free Kolmogorov complexity (i.e., $\K(A \uh n) \leq \K(n)+O(1)$). Lowness has been studied for other randomness notions. The work of Terwijn and Zambella~\cite{TerwijnZ2001} and of Kjos-Hanssen et al.~\cite{Kjos-HanssenNS2005} characterized low for Schnorr randomness as computably traceable (a strengthening of hyperimmune-free; see below). Nies~\cite{Nies2005} showed that only computable oracles can be low for computable randomness.

One can also study lowness for a pair of randomness notions. If $\mathscr{R}$ and $\mathscr{S}$ are two randomness notions with $\mathscr{R} \subseteq \mathscr{S}$, $\Lowpair{\mathscr{R}}{\mathscr{S}}$ is the set of oracles~$A$ such that $\mathscr{R} \subseteq \mathscr{S}^A$. The task of characterizing the sequences that are low for randomness has attracted a lot of effort in the last few years, and is now nearly completed, as shown in the following diagram.

%{
%\scriptsize
%\newcommand{\multilabel}[1]{\begin{tabular}{@{}c@{}}#1\end{tabular}}
%
%\begin{center}
%\begin{tabular}{c|c|c|c|c|c|}
%\multicolumn{6}{c}{\hspace{1.5cm}$\mathcal{N}$} \\
%\multicolumn{1}{c}{}
%	& \multicolumn{1}{c}{DR}
%	& \multicolumn{1}{c}{W2R}
%	& \multicolumn{1}{c}{MLR}
%	& \multicolumn{1}{c}{CR}
%	& \multicolumn{1}{c}{SR} \\\hhline{~-~---}
%\multicolumn{1}{r|}{DR}
%	& \cellcolor[gray]{.85}computable
%	& 
%	& \cellcolor[gray]{.85}$\K$-trivial
%	& \cellcolor[gray]{.85}$\K$-trivial
%	& \cellcolor[gray]{.85}\multilabel{c.e.\\[-6pt]traceable} \\\hhline{~-----}
%\multicolumn{1}{c}{} & \multicolumn{1}{r|}{W2R}
%	& $\K$-trivial \cite{DowneyNWY2006,KMS,Nies2009}
%	& $\K$-trivial \cite{DowneyNWY2006}
%	& $\K$-trivial \cite{Nies2009}
%	& \cellcolor[gray]{.85}\multilabel{c.e.\\[-6pt]traceable} \\\hhline{~~----}
%\multicolumn{2}{c}{} & \multicolumn{1}{r|}{MLR}
%	& $\K$-trivial \cite{Nies2005}
%	& $\K$-trivial \cite{Nies2005}
%	& \multilabel{c.e.\\[-6pt]traceable} \cite{Kjos-HanssenNS2005} \\\hhline{~~~---}
%\multicolumn{3}{r}{$\mathcal{M}$\hspace{.5cm}\ } & \multicolumn{1}{r|}{CR}
%	& computable \cite{Nies2005}
%	& \multilabel{computably\\[-6pt]traceable} \cite{Kjos-HanssenNS2005} \\\hhline{~~~~--}
%\multicolumn{4}{c}{} & \multicolumn{1}{r|}{SR}
%	& \multilabel{computably\\[-6pt]traceable} \cite{TerwijnZ2001,Kjos-HanssenNS2005} \\\hhline{~~~~~-}
%\end{tabular}
%
%\footnotesize $\Lowpair{\mathcal{M}}{\mathcal{N}}$ for various randomness classes.
%\end{center}
%}

{
\scriptsize
\newcommand{\multilabel}[1]{\begin{tabular}{@{}c@{}}#1\end{tabular}}

\begin{center}
\begin{tabular}{c|c|c|c|c|}
\multicolumn{5}{c}{\hspace{1.5cm}$\mathscr{S}$} \\
\multicolumn{1}{c}{}
	& \multicolumn{1}{c}{W2R}
	& \multicolumn{1}{c}{MLR}
	& \multicolumn{1}{c}{CR}
	& \multicolumn{1}{c}{SR} \\\hhline{~----}
\multicolumn{1}{r|}{W2R}
	& $\K$-trivial \cite{DowneyNWY2006,KMS,Nies2009}
	& $\K$-trivial \cite{DowneyNWY2006}
	& $\K$-trivial \cite{Nies2009}
	& \cellcolor[gray]{.85}\multilabel{c.e.\\[-6pt]traceable} \\\hhline{~----}
\multicolumn{1}{c}{} & \multicolumn{1}{r|}{MLR}
	& $\K$-trivial \cite{Nies2005}
	& $\K$-trivial \cite{Nies2005}
	& \multilabel{c.e.\\[-6pt]traceable} \cite{Kjos-HanssenNS2005} \\\hhline{~~---}
\multicolumn{2}{r}{$\mathscr{R}$\hspace{.5cm}\ } & \multicolumn{1}{r|}{CR}
	& computable \cite{Nies2005}
	& \multilabel{computably\\[-6pt]traceable} \cite{Kjos-HanssenNS2005} \\\hhline{~~~--}
\multicolumn{3}{c}{} & \multicolumn{1}{r|}{SR}
	& \multilabel{computably\\[-6pt]traceable} \cite{TerwijnZ2001,Kjos-HanssenNS2005} \\\hhline{~~~~-}
\end{tabular}

\footnotesize $\Lowpair{\mathscr{R}}{\mathscr{S}}$ for various randomness classes.
\end{center}
}

Note that each class in the diagram is contained in the classes above it and to its right. The gray entry is settled in this paper. %Note that there are no squares for $\Lowpair{\DR}{\SMLR}$ or $\Lowpair{\SMLR}{\DR}$; as Demuth randomness and weak $2$-randomess are incomparable, neither of these classes is well-defined.
It should be also noted that this diagram omits the results obtained by Greenberg and Miller~\cite{GreenbergM2009} that characterize almost all lowness notions related to \emph{weak $1$-randomness}, as we do not discuss weak $1$-randomness in the present paper.

Although we will not directly use these notions in this paper, we recall the definitions of the classes that are referred to in this diagram. We defined $K$-triviality above. A sequence $A$ is \emph{computably traceable} if there exists a single computable function~$h$ such that for any total function $f: \N \rightarrow \N$ computable in~$A$, there exists a uniformly computable sequence of finite sets $(T_n)_{n \in \N}$, given by their strong index, such that for all~$n$, $f(n) \in T_n$ and $|T_n| < h(n)$. The definition of \emph{c.e.\ traceability} is the same, except that the sets $T_n$ are given by their index as c.e.\ sets. Kjos-Hanssen et al.~\cite{Kjos-HanssenNS2005} introduced c.e.\ traceability specifically to characterize $\Lowpair{\MLR}{\SR}$, making it one of several examples of interesting computability theoretic properties that have arisen from the study of randomness and lowness notions.

%---------------------------------------------------------
\section{Testing randomness via open covers}\label{sec:open-covers}
%---------------------------------------------------------

In this section, we present an alternative way to look at the above randomness notions. Instead of using tests, i.e., sequences of open sets, it is possible to provide equivalent definitions involving a single open set (or c.e.\ set of strings). The first theorem below is due to Ku{\v c}era~\cite{Kucera1985} and characterizes Martin-L\"of randomness. We prove analogous theorems for computable and Schnorr randomness.

\begin{thm} \label{thm:mlr-covers}
Let $X \in \cs$. The following are equivalent:\\
(i) $X$ is not Martin-L\"of random\\
(ii) There is a bounded c.e.\ open set $\mathcal{U}$ such that all tails of $X$ belong to $\mathcal{U}$.\\
(iii) $X \in U^\omega$ for some bounded c.e.\ prefix-free subset $U$.
\end{thm}
\begin{proof}
$(i) \Rightarrow (ii)$ This follows easily from the existence of a universal Martin-L\"of test $(\mathcal{U}_n)_{n \in \N}$. Let $\mathcal{U}=\mathcal{U}_1$. So $\mathcal{U}$ is a bounded c.e.\ open set covering all non-Martin-L\"of random sequences. If $X$ is not Martin-L\"of random, then none of its tails are Martin-L\"of random. Hence they all belong to $\mathcal{U}$.

$(ii) \Rightarrow (iii)$ Suppose that $\mathcal{U}$ is a bounded c.e.\ open set and that all of the tails of $X$ belong to $\mathcal{U}$. Let $U\subset \fs$ be a c.e.\ prefix-free set such that $\mathcal{U} = [U]$. We show by induction on $n$ that $X\in [U^n]$. This is true for $n=1$ by assumption. Now assume that $X\in [U^n]$. So for some $\sigma\in U^n$, there is a $Z$ such that $X=\sigma Z$. Since $Z$ is a tail of $X$, we know that $Z\in\mathcal{U} = [U]$. But this implies that $X\in [U^{n+1}]$. Thus $X\in [U^n]$ for all $n$. Therefore, $X\in U^\omega = \bigcap_n [U^n]$.

$(iii) \Rightarrow (i)$ Assume that $X \in U^\omega$ for some bounded c.e.\ prefix-free subset $U$. For each $n$, we have $X\in \mathcal{U}_n = [U^n]$. Also, $\mu(\mathcal{U}_n) = \mu(U^n) = \mu(U)^n$, so $(\mathcal{U}_n)_{n \in \N}$ is a Martin-L\"of test (in the more general sense of Remark~\ref{rem:test-note}). Since $(\mathcal{U}_n)_{n \in \N}$ covers $X$, it is not Martin-L\"of random.
\end{proof}

\begin{thm} \label{thm:sr-covers}
Let $X \in \cs$. The following are equivalent:\\
(i) $X$ is not Schnorr random.\\
(ii) There is a bounded Schnorr open set~$\mathcal{U}$ such that all tails of $X$ belong to~$\mathcal{U}$.\\
(iii) $X \in U^\omega$ for some bounded Schnorr prefix-free subset~$U$ of $\fs$.
\end{thm}

\begin{proof}
$(i) \Rightarrow (ii)$ Suppose that $X$ is not Schnorr random. Then $X \in \bigcap_n \mathcal{V}_n$ where $(\mathcal{V}_n)_{n \in \N}$ is a Schnorr test (say with $\mu(\mathcal{V}_n)=2^{-n}$). We build the desired set~$\mathcal{U}$ from this Schnorr test. For every~$k$, the tail $Y_k=X(k)X(k+1)\ldots$ of $X$ belongs to $\bigcap_n (\mathcal{V}_n \mid \tau)$, where $\tau=X(0)...X(k-1)$. Thus, for all~$n$, $Y_k$ belongs to
\[
\bigcup_{\substack{\sigma \\ |\sigma|=k}} (\mathcal{V}_n \mid \sigma),
\]
which, for~$n$ large enough, has small measure. For example, for $n=3k+2$, the above set has measure at most $2^{-k-2}$ (indeed each of the $2^k$ sets of type $(\mathcal{V}_n \mid \sigma)$ has measure at most $\mu(\mathcal{V}_n)/\mu([\sigma]) \leq 2^{-n}2^k$, hence the total measure is at most $2^{-n+2k}$).
Thus, define:
\[
\mathcal{U}=\bigcup_{k \in \N} \bigcup_{\substack{\sigma \\ |\sigma|=k}} (\mathcal{V}_{3k+2} \mid \sigma).
\]
We claim that~$\mathcal{U}$ is as wanted. Indeed, $\mathcal{U}$ is clearly $\Sigma^0_1$. By the above discussion, $\mathcal{U}$ contains all tails of~$X$, the measure of $\mathcal{U}$ is at most
\[
\sum_{k \in \N} 2^{-k-2} \leq 1/2.
\]
To see that the measure of $\mathcal{U}$ is computable, note that the measure $\mu(\mathcal{V}_{3k+2} \mid \sigma)$ is computable uniformly in $k$ and $\sigma$. So the measure of
\[
\bigcup_{k \leq N} \bigcup_{\substack{\sigma \\ |\sigma|=k}} (\mathcal{V}_{3k+2} \mid \sigma),
\]
is computable, uniformly in $N$, and approximates $\mu(\mathcal{U})$ up to $\sum_{k>N} 2^{-k-2} < 2^{-N}$.\\

%$(ii) \Rightarrow (iii)$. Suppose $(ii)$ holds and let~$U$ be the c.e.\ prefix-free set of strings generating~$\mathcal{U}$ (hence the measure of~$U$ is $\mu(\mathcal{U})<1$. Since $X \in \mathcal{U}$, there exists $\sigma_0 \in 

The proofs of $(ii) \Rightarrow (iii)$ and $(iii) \Rightarrow (i)$ go exactly as in Theorem~\ref{thm:mlr-covers}.\\

\end{proof}

\begin{thm} \label{thm:cr-covers}
Let $X \in \cs$. The following are equivalent:\\
(i) $X$ is not computably random.\\
(ii) There exists a winning open set $\mathcal{U}$ such that all tails of $X$ belong to~$\mathcal{U}$.\\ 
(iii) $X \in U^\omega$ for some winning subset $U$ of $\fs$.
\end{thm}

\begin{proof}
$(i) \Rightarrow (ii)$. Suppose $X$ is not computably random. Then there exists a computable martingale~$d$ that succeeds against~$X$. Up to adding a positive constant to~$d$ (preserving the fairness condition) we can assume that~$d$ is positive. Now, for all strings~$\sigma$, one can consider the ``translated'' version $d_\sigma$ of $d$ defined by
\[
d_\sigma(\tau)=d(\sigma \tau).
\]
It is easy to see that if $Y$ is a tail of $X$, with $X=\sigma Y$, then $d_\sigma$ succeeds against~$Y$. Therefore, the martingale~$D$ defined by:
\[
D(\tau)=\sum_{\sigma \in \fs} 2^{-2|\sigma|-1}\cdot  \frac{d_\sigma(\tau)}{d_\sigma(\emptystring)}
\]
succeeds against all tails of~$X$. Moreover, $D$ is normed and computable as a sum of exponentially decreasing uniformly computable terms. By the result of Lutz~\cite{Lutz2004} mentioned earlier, we can also assume that~$D$ is exactly computable. Since~$D$ succeeds against all tails of $X$, this in particular implies that all tails of~$X$ belong to the winning open set $[U]$ with $U=\{\sigma\; : \; \sigma~\text{minimal s.t. } D(\sigma) \geq 2\}$.

The proof of $(ii) \Rightarrow (iii)$ is as in Theorem~\ref{thm:mlr-covers}.

%$(ii) \Rightarrow (iii)$. This is straightforward from the definition of a winning open set. Let~$U$ be the winning subset of $\fs$ such that all tails of~$X$ belong to~$[U]$. We can therefore recursively construct a sequence $\sigma_0,\sigma_1,\sigma_2,\ldots$ of elements of~$U$ whose concatenation (in that order) is~$X$: having found $\sigma_0\sigma_1\ldots\sigma_n$ prefix of $X$ where $\sigma_i \in U$ for all~$i$, we know that the tail of~$X$ coming after $\sigma_0\sigma_1\ldots\sigma_n$ is in $[U]$, hence some prefix~$\sigma_{n+1}$ of this tail belongs to~$U$, and $\sigma_0\sigma_1\ldots\sigma_n\sigma_{n+1}$ is a prefix of~$X$. This proves that~$X \in U^\omega$. 

$(iii) \Rightarrow (i)$. If $X \in U^\omega$ where $U$ is a winning set of strings, let $d$ be the normed exactly computable martingale and $q>1$ a rational such that $U$ is a $(d,q)$-winning set. We can assume that $d$ is positive, otherwise we set $d'=\frac{1}{2}d+\frac{1}{2}$, which makes $U$ a $(d',\frac{1+q}{2})$-winning set, and $d'$ is positive. Now, we design a computable martingale~$D$ that succeeds on all sequences in $U^\omega$. Basically, $D$ simulates the martingale~$d$ and ``resets'' after reading a block $\sigma \in U$. Formally, $D$ is defined by induction. We set $D(\emptystring)=1$ and if $D(\sigma)$ is already defined, we write $\sigma=\rho\tau$ where $\rho$ is a concatenation of strings in~$U$ and~$\tau$ has not prefix in~$U$ (this decomposition is unique as $U$ is prefix-free) and then set
\[
D(\sigma \iota)= D(\sigma) \cdot \frac{d(\tau \iota)}{d(\tau)},
\]
for $\iota \in \{0,1\}$. It is easy to see that $D$ is an exactly computable martingale, and if $\sigma$ is a concatenation of $k$ strings in $U$, $D(\sigma) \geq q^k$. Hence $D$ succeeds against all sequences in $U^\omega$. 
\end{proof}

\begin{rem} \label{rem:tests-and-covers}
In fact, the proofs of Theorems~\ref{thm:mlr-covers}, \ref{thm:sr-covers} and \ref{thm:cr-covers} show a little more. %Indeed, these three randomness notions are described via tests (resp. Martin-L\"of tests, tests induced by computable martingales, Schnorr tests). 
What we actually proved is the following equivalence for any subset $\mathcal{X}$ of $\cs$\\
(i) $\mathcal{X}$ is covered by a Martin-L\"of test (resp. a test induced by a martingale, a Schnorr test).\\
(ii) There exists a single bounded c.e.\ open set (resp. winning open set, bounded Schnorr open set) $\mathcal{U}$ such that for any $X \in \mathcal{X}$, all tails of $X$ are in $\mathcal{U}$.\\
(iii) There exists a single bounded c.e.\ set of strings (resp. winning set of strings, bounded Schnorr set of strings) $U$ such that $\mathcal{X} \subseteq U^\omega$.
\end{rem}

%--------------------------------------------------------
\section{The main lemma}\label{sec:main-lemma}
%--------------------------------------------------------

The following technical lemma is the cornerstone of this paper. It lets us use the characterizations of Martin-L\"of randomness, computable randomness and Schnorr randomness proven in the previous section to study the associated lowness notions. Roughly speaking, it states that if a prefix-free set of strings~$U$ is not covered by any of the members of a (reasonably well-behaved) collection~$\mathbf{C}$ of open sets, then there exists an $X \in U^\omega$ that passes all tests that can be built from the elements of~$\mathbf{C}$. 

\begin{lem}\label{lem:main}
Let $\mathbf{C}$ be a class of bounded open subsets of $\cs$. Let also $(\mathcal{T}^{(e)}_n)_{e,n \in \N}$ be a countable family of tests (i.e., for all~$e$, $(\mathcal{T}^{(e)}_n)_{n \in \N}$ is a test) such that~$\mathcal{T}^{(e)}_n$ belongs to $\mathbf{C}$ for all $e,n$. Suppose we have the following closure properties.\\
(P1) For all ~$\mathcal{U} \in \mathbf{C}$ and $\sigma \in \fs$, if $\mu(\mathcal{U} \mid \sigma) < 1$, then there exists a~${\mathcal{V}\in \mathbf{C}}$ such that $(\mathcal{U} \mid \sigma) \subseteq \mathcal{V}$.\\
(P2) For all ~$\mathcal{U} \in \mathbf{C}$, there exists a~$\mathcal{V}\in \mathbf{C}$ such that $\mathcal{U} \subseteq \mathcal{V}$, and for all~$\sigma \in \fs$, if $\mu(\mathcal{U} \mid \sigma)=1$, then $[\sigma] \subseteq \mathcal{V}$.\\
(P3) For all $\mathcal{U} \in \mathbf{C}$, and $\sigma \in \fs$, if $\mu(\mathcal{U} \mid \sigma) < 1$, then for all $e \in \N$, there exists $n_e \in \N$ and $\mathcal{V} \in \mathbf{C}$ such that $(\mathcal{U} \cup \mathcal{T}^{(e)}_{n_e}) \subseteq \mathcal{V}$ and $\mu(\mathcal{V} \mid \sigma)<1$.

Finally, let $W$ be a prefix-free subset of $\fs$ such that $[W]$ cannot be covered by any open set $\mathcal{U} \in \mathbf{C}$.

Then, there exists $X \in W^\omega$ that passes all tests $\mathcal{T}^{(e)}$. 
\end{lem}

\begin{proof}
We build~$X$ by a finite extension technique: having built a prefix $\sigma_e$ of $X$, during stage~$e$, we build an extension $\sigma_{e+1}$ of $\sigma_e$. This is done as follows. We begin with $\sigma_0$ equal to the empty string, and $\mathcal{U}_0$ equal to the empty subset of $\cs$. At the beginning of stage~$e$, suppose we have already built $\sigma_e$ and $\mathcal{U}_e$, satisfying $\mu(\mathcal{U}_e \mid \sigma_e) < 1$. We then use property~$(P3)$ to find $n_e$ and $\mathcal{V} \in \mathbf{C}$ such that $(\mathcal{U}_e \cup \mathcal{T}^{(e)}_{n_e}) \subseteq \mathcal{V}$ and $\mu(\mathcal{V} \mid \sigma_e)<1$. Next, we pick a non-empty string $\tau \in W$ such that $\mu(\mathcal{V} \mid \sigma_e \tau) < 1$. We then set $\mathcal{U}_{e+1}=\mathcal{V}$ and $\sigma_{e+1}=\sigma_e\tau$, finishing the $e$-th stage.

It remains to verify that this construction works. First, we have to make sure that at stage~$e$ of our construction, there is indeed a string $\tau \in W$ such that $\mu(\mathcal{V} \mid \sigma_e \tau) < 1$. Suppose that this is not the case. This means that  $\mu(\mathcal{V} \mid \sigma_e \tau) = 1$ for all~$\tau \in W$, or equivalently that $\mu\big((\mathcal{V} \mid \sigma_e) \mid \tau\big) = 1$ for all~$\tau \in W$. By definition of~$\mathcal{V}$, we have $\mu(\mathcal{V} \mid \sigma_e)<1$, so we can apply property~$(P1)$ to get some $\mathcal{V}' \in \mathbf{C}$ such that~$(\mathcal{V} \mid \sigma_e) \subseteq \mathcal{V'}$. In particular, $\mathcal{V}'$ is such that $\mu(\mathcal{V}' \mid \tau) = 1$ for all~$\tau \in W$. Now by property~$(P2)$ there exists $\mathcal{V}'' \in \mathbf{C}$ covering~$\mathcal{V}'$ and such that $[\tau] \subseteq \mathcal{V}''$ for all $\tau \in W$, which means that $[W]$ is covered by~$\mathcal{V}''$. This contradicts the assumption that~$[W]$ is not covered by a set that belongs to~$\mathbf{C}$.

Now, let $X$ be the unique element of $\cs$ such that all $\sigma_e$'s are prefixes of~$X$. It is easy to see from the construction that $X \in W^\omega$. Moreover, suppose $X$ fails a test $\mathcal{T}^{(e)}$. This would imply that $X \in  \mathcal{T}^{(e)}_{n_e}$ (the $n_e$ being defined in the above construction). Thus, there would exist $e'>e$ large enough, such that $[\sigma_{e'}] \subseteq \mathcal{T}^{(e)}_{n_e}$. This would be a contradiction since, by construction, on the one hand $[\sigma_{e'}] \nsubseteq \mathcal{U}_{e'+1}$ and on the other hand $\mathcal{T}^{(e)}_{n_e} \subseteq \mathcal{U}_{e+1} \subseteq \mathcal{U}_{e'+1}$. 
\end{proof}

\begin{prop} \label{prop:closure}
The hypotheses (P1,P2,P3) of Lemma~\ref{lem:main} are satisfied in the following three cases.\\
(MLR) $\mathbf{C}$ is the class of bounded c.e.\ open sets and $(\mathcal{T}_e)$ is the family of Martin-L\"of tests.\\
(CR) $\mathbf{C}$ is the class of open sets that are winning sets of exactly computable martingales and $(\mathcal{T}_e)$ is the family of tests induced by those martingales.\\
(SR) $\mathbf{C}$ is the class of bounded Schnorr open sets and $(\mathcal{T}_e)$ is the family of Schnorr tests.
\end{prop}

\begin{proof}
(MLR) As we previously observed, for every c.e.\ open set $\mathcal{U}$ and $\sigma \in \fs$, $(\mathcal{U} \mid \sigma)$ is a c.e.\ open set and its index can be computed from an index of $\mathcal{U}$ and $\sigma$. The property~(P1) thus follows immediately. For the property~(P3), given a c.e.\ open set $\mathcal{U}$ and $\sigma \in \fs$, such that $\mu(\mathcal{U} \mid \sigma) < 1- 2^{-k}$ for some $k >0$, together with a Martin-L\"of test~$\mathcal{T}^{(e)}$, take $n_e=|\sigma|+k$. By the definition of a Martin-L\"of test, $\mu(\mathcal{T}^{(e)}_{n_e}) < 2^{-n_e} = 2^{-|\sigma|-k}$. So $\mu(\mathcal{T}^{(e)}_{n_e} \mid \sigma) < 2^{-k}$, and the set $\mathcal{V}=\mathcal{U} \cup \mathcal{T}^{(e)}_{n_e}$ is c.e.\ open and satisfies $\mu(\mathcal{V} \mid \sigma) < (1-2^{-k})+2^{-k} < 1$. We now check that~(P2) holds. Given a bounded c.e.\ open set~$\mathcal{U}$, let~$q<1$ be a rational such that $\mu(\mathcal{U})<q$ and set 
\[
\mathcal{V} = \bigcup \big{\{}[\sigma]\; : \; \mu(\mathcal{U} \mid \sigma) > q \big{\}}
\]
It is clear that~$\mathcal{V}$ is c.e.\ open and if $\mu(\mathcal{U} \mid \sigma)=1$ then $[\sigma] \subseteq \mathcal{V}$. It remains to check that~$\mathcal{V}$ is bounded. Let $F$ be the set of strings $\sigma$ that are minimal among those satisfying $\mu(\mathcal{U} \mid \sigma) > q$. We have
\[
\mu(\mathcal{V})=\sum_{\sigma \in F} \mu([\sigma]) \leq \sum_{\sigma \in F} \frac{\mu(\mathcal{U} \cap [\sigma])}{q} \leq \frac{\mu(\mathcal{U} \cap [F])}{q} \leq \frac{\mu(\mathcal{U})}{q} < 1.
\]

(CR) Let $U$ be a c.e.\ set of strings such that $U=\{\sigma\; : \; \sigma~\text{minimal s.t. } d(\sigma) \geq q\}$ for some exactly computable normed martingale~$d$ and rational~${q>1}$. 

For property~(P1), suppose that $\mu(U \mid \sigma) < 1$. We thus have $d(\sigma) < q$ (otherwise $[\sigma] \subseteq [U]$), and we can also assume that $d(\sigma)>0$ (otherwise $(U \mid \sigma)=\emptyset$ and there is nothing to prove). We have
\begin{align*}
(U \mid \sigma) &=\{ \tau \; : \; \tau~\text{minimal s.t. } d(\sigma\tau) \geq q\} \\
&= \Big{\{} \tau \; : \; \tau~\text{minimal s.t. } \frac{d(\sigma\tau)}{d(\sigma)} \geq \frac{q}{d(\sigma)}\Big{\}}.
\end{align*}
It is easy to check that $\tau \mapsto \frac{d(\sigma\tau)}{d(\sigma)}$ is an exactly computable normed martingale, and since $\frac{q}{d(\sigma)} > 1$ we see that $(U \mid \sigma)$ is a winning set of strings. 

For property~(P2), we will see that the property $\big( \mu(U \mid \sigma)=1 \Rightarrow [\sigma] \subseteq [U] \big)$ holds. Indeed, if $\mu(U \mid \sigma)=1$, this means that for almost all~$X \in [\sigma]$, there exists an~$n$ such that $d(X \uh n) \geq q$. By the Ville-Kolmogorov inequality, this implies that $d(\sigma) \geq q$. 

For property~(P3), let $\sigma$ be such that $\mu(U \mid \sigma)<1$. As we have seen, this implies $d(\sigma) < q$. Take the $e$-th test~$\mathcal{T}^{(e)}$ associated to an exactly computable normed martingale~$d_e$ (i.e., $\mathcal{T}^{(e)}_n$ is the open set generated by the strings $\sigma$ such that $d(\sigma) \geq 2^n$). We need to find $n_e$ such that $([U] \cup \mathcal{T}^{(e)}_{n_e})$ is covered by $[V]$, where $V$ is a winning set such that $\mu(V \mid \sigma) < 1$. Let~$n_e$ be large, to be specified later. Let $D$ be the exactly computable normed martingale defined by
\[
D= (1-2^{-n_e+1})d +2^{-n_e+1} d_e.
\]
We have $D(\sigma)=(1-2^{-n_e+1})d(\sigma) +2^{-n_e+1} d_e(\sigma)$. Now, suppose $X \in ([U] \cup \mathcal{T}^{(e)}_{n_e})$. If $X \in [U]$, then $d(X \uh n) \geq q$ for some~$n$, and then $D(X \uh n) \geq (1-2^{-n_e+1})q$. If $X \in  \mathcal{T}^{(e)}_{n_e}$, then $d(X \uh n) \geq 2^{-n_e+1} 2^{n_e}=2$ for some~$n$. We thus consider the set 
\[
V=\{ \tau\; : \; \tau~\text{minimal s.t. } D(\tau) \geq \min((1-2^{-n_e+1})q,2)\}
\]
Now, for $n_e$ large enough, we can ensure from the previous calculations that $D(\sigma)$ is as close as we want to $d(\sigma)<q$, and thus that $\min((1-2^{-n_e+1})q,2)>D(\sigma)$. The above set $V$ is then as desired: it is clearly a winning set, it covers $([U] \cup \mathcal{T}^{(e)}_{n_e})$, and $\mu(V \mid \sigma)<1$ as $D(\sigma) < \min((1-2^{-n_e+1})q,2)$. \\

(SR) Property~(P1) is clearly satisfied. Given $\mathcal{U}$ a c.e.\ open set of computable measure, and $\sigma \in \fs$, $(\mathcal{U} \mid \sigma)$ is a c.e.\ open set whose measure is computable (uniformly in $\sigma$ and an index for $\mathcal{U}$). Indeed, if $\widehat{\mathcal{U}}$ is a clopen approximation of~$\mathcal{U}$ such that $\mu(\mathcal{U} \setminus \widehat{\mathcal{U}}) < \varepsilon$, then $\mu\Big((\mathcal{U} \mid \sigma) \setminus (\widehat{\mathcal{U}} \mid \sigma)\Big) < \varepsilon\cdot 2^{|\sigma|}$. Property~(P3) is also clearly satisfied, as it is easy to see that given two c.e.\ open sets $\mathcal{U}$ and $\mathcal{V}$ of computable measure, $\mathcal{U} \cup \mathcal{V}$ is also a c.e.\ open set of computable measure.

For property~(P2), let $\mathcal{U}$ be a bounded Schnorr open set and let $k$ be large enough that $\mu(\mathcal{U})<1-2^{-k}$. We define a computable set of strings $V$ such  $[V]\supseteq\mathcal{U}$ is a bounded Schnorr open set. For every $\sigma\in\fs$, look for a stage $s$ such that $\mu(\mathcal{U}\setminus\mathcal{U}_s)<2^{-2|\sigma|-k-1}$. If $\mu(\mathcal{U}_s \mid \sigma)>1-2^{-|\sigma|-k-1}$, then put $\sigma$ into~$V$. It is clear that~$V$ is computable. If $\mu(\mathcal{U} \mid \sigma)=1$, then it must be the case that $\mu(\mathcal{U}_s \mid \sigma)>1-2^{-|\sigma|-k-1}$; otherwise
\[
\mu(\mathcal{U}\setminus\mathcal{U}_s)\geq \mu([\sigma]\setminus\mathcal{U}_s)\geq 2^{-|\sigma|} - 2^{-|\sigma|}(1-2^{-|\sigma|-k-1}) = 2^{-2|\sigma|-k-1},
\]
which contradicts the choice of $s$. Therefore, $\mu(\mathcal{U} \mid \sigma)=1$ implies that $\sigma\in V$, so $[\sigma]\subseteq[V]$ as required. This also implies that $\mathcal{U}\subseteq [V]$. It remains to show that $[V]$ is bounded and has computable measure. By adding $\sigma$ to $V$, we are increasing the measure of $[V]\setminus\mathcal{U}$ by less than $2^{-|\sigma|} 2^{-|\sigma|-k-1} = 2^{-2|\sigma|-k-1}$. Therefore,
\[
\mu([V]\setminus\mathcal{U}) < \sum_{\sigma\in\fs} 2^{-2|\sigma|-k-1} = \sum_{n\in\omega} 2^n2^{-2n-k-1} = 2^{-k}.
\]
This implies that $\mu([V]) < \mu(\mathcal{U}) + 2^{-k} < 1$, so $[V]$ is a bounded c.e.\ open set. Similarly, for every $m$,
\[
\mu\big([V]\setminus([V\cap 2^{<m}]\cup\mathcal{U})\big) < \sum_{\sigma\in 2^{\geq m}} 2^{-2|\sigma|-k-1} = \sum_{n\geq m} 2^n2^{-2n-k-1} = 2^{-m-k}.
\]
But $\mu([V\cap 2^{<m}]\cup\mathcal{U})$ is computable uniformly in $m$, so $\mu([V])$ is also computable. 
\end{proof}

%To check property~$(P2)$, we proceed as in (MLR): given a bounded computable open set~$\mathcal{U}$, let~$q<1$ be a rational such that $\mu(\mathcal{U})<q$ and set 
%\[
%\mathcal{V} = \bigcup \big{\{}[\sigma]\; : \; \mu(\mathcal{U} \mid \sigma) > q \big{\}}
%\]
%We have already seen that $\mathcal{V}$ is a c.e.\ open set, is bounded, and if $\mu(\mathcal{U} \mid \sigma)=1$ then $[\sigma] \subseteq \mathcal{V}$. It remains to prove that~$\mu(\mathcal{V})$ is computable. Given $\varepsilon>0$ small enough to have $q<1-\varepsilon$, we can compute an $(\varepsilon/q)$-approximation of $\mu((\mathcal{V})$ as follows. First, effectively find an integer~$N$ such that for all $\sigma$ of length~$n$, either $\mu(\mathcal{U} \mid \sigma) > 1- \varepsilon$ or $\mu(\mathcal{U} \mid \sigma) < \varepsilon$. Such an~$N$ exists by Lebesgue density theorem, and can be found effectively as  $\mu(\mathcal{U} \mid \sigma)$ is computable uniformly in~$\sigma$. Now, call $W=\{\sigma\; : \; |\sigma|=N \wedge \mu(\mathcal{U} \mid \sigma) > 1- \varepsilon\}$. We have $[W] \subseteq \mathcal{V}$ and for $\sigma$ of length~$N$ that is not in $W$:
%\[
%\mu(\mathcal{V} \mid \sigma) \leq \frac{\mu(\mathcal{U})}{q} < \frac{\varepsilon}{q}.
%\]
%Thus:
%\[
%\mu([W]) \leq \mu(\mathcal{V}) \leq \mu([W]) + \frac{\varepsilon}{q}.
%\]
%\end{proof}

%--------------------------------------------------------
\section{Applications}\label{sec:applications}
%--------------------------------------------------------

\subsection{Lowness notions}\label{ssec:apps-lowness}

Together with the results of Section~\ref{sec:open-covers}, Lemma~\ref{lem:main} has interesting consequences. First, it follows from it that all lowness notions involving Martin-L\"of randomness, computable randomness, or Schnorr randomness can essentially be reduced to a property of open sets. For example $A \in \Low{\MLR}$ if and only if every bounded $A$-c.e.\ open set can be covered by a bounded c.e.\ open set (under this form, this was first stated by Kjos-Hanssen in~\cite{Kjos-Hanssen2007}, but most of the ideas are already present in Ku{\v c}era~\cite{Kucera1985}); $A \in \Low{\CR}$ if and only if every $A$-winning open set can be covered by a winning open set; $A \in \Low{\SR}$ if and only if every bounded $A$-computable open set can be covered by a bounded computable open set. This also works for pairs of randomness notions: for example, $A \in \Lowpair{\MLR}{\CR}$ if and only if every $A$-winning open set can be covered by a bounded c.e.\ open set. Let us briefly present the proof of one such result (the proofs of all the other claims are almost identical).

\begin{cor} \label{cor:low-cr-sr}
Let $A \in \cs$. The following are equivalent:\\
(i) $A \in \Lowpair{\CR}{\SR}$\\
(ii) Every bounded $A$-Schnorr open set can be covered by a winning open set.  
\end{cor}   

\begin{proof}
$(i) \Rightarrow (ii)$.  Suppose that $(ii)$ does not hold, i.e., there exists a bounded $A$-Schnorr open set $\mathcal{U}$ that cannot be covered by any winning open set. Let~$U$ be an $A$-c.e.\ prefix-free set of strings generating~$\mathcal{U}$. We can apply Lemma~\ref{lem:main} with $\mathbf{C}$ the class of winning open sets and $\mathcal{T}$ the family tests induced by exactly computable martingales (this is allowed by Proposition~\ref{prop:closure}), from which we get the existence of $X \in U^\omega$ that passes all tests induced by exactly computable martingales (hence $X$ is computably random). By Theorem~\ref{thm:sr-covers} (relativized to $A$), $X$ is not $A$-Schnorr random. Thus $A \notin \Lowpair{\CR}{\SR}$. 

$(ii) \Rightarrow (i)$. Suppose $(ii)$ holds and take $X \in \cs$ such that $X$ is not $A$-Schnorr random. By Theorem~\ref{thm:sr-covers}, there exists a bounded $A$-Schnorr open set $\mathcal{U}$ such that all tails of~$X$ belong to~$\mathcal{U}$. By assumption~$(ii)$, $\mathcal{U}$ can be covered by a winning open set~$\mathcal{V}$. Thus, all tails of $X$ belong to $\mathcal{V}$, which by Theorem~\ref{thm:cr-covers} implies that~$X$ is not computably random.  
\end{proof}

\subsection{Weak reducibilities}\label{ssec:apps-reducibilities}

\newcommand{\LR}{\mathit{LR}}
\newcommand{\CRwr}{\mathit{CR}}

Nies~\cite{Nies2005} introduced a weak reducibility generalizing low for Martin-L\"of randomness. He defined $A\leq_\LR B$ to mean that $\MLR^B\subseteq\MLR^A$. So, $A\leq_\LR \emptyset$ iff $A\in\Low{\MLR}$. What Kjos-Hanssen~\cite{Kjos-Hanssen2007} actually proved was that $A\leq_\LR B$ if and only if every bounded $A$-c.e.\ open set can be covered by a bounded $B$-c.e.\ open set. This result follows easily from our framework, as do the analogous results for the weak reducibilites associated with computable and Schnorr randomness, although these relations have not received attention. For example, if we write $A\leq_\CRwr B$ to mean that $\CR^B\subseteq\CR^A$, then the following characterization follows by a proof identical to that of Corollary~\ref{cor:low-cr-sr}.

\begin{cor}
The following are equivalent for $A,B\in \cs$:\\
(i) $A\leq_\CRwr B$\\
(ii) Every $A$-winning open set can be covered by a $B$-winning open set.  
\end{cor}   

%\begin{proof}
%$(i) \Rightarrow (ii)$.  Suppose that $(ii)$ does not hold. Thus there is an $A$-winning open set $\mathcal{U}$ that cannot be covered by any $B$-winning open set. Let~$U$ be an $A$-c.e.\ prefix-free set of strings generating~$\mathcal{U}$. We can apply Lemma~\ref{lem:main} with $\mathbf{C}$ the class of $B$-winning open sets and $\mathcal{T}$ the family tests induced by exactly $B$-computable martingales. This is allowed by the Proposition~\ref{prop:closure} relativized to $B$. Lemma~\ref{lem:main} gives us an $X \in U^\omega$ passing all tests induced by exactly $B$-computable martingales (hence $X$ is $B$-computably random). By Theorem~\ref{thm:sr-covers} relativized to $A$, $X$ is not $A$-computably random. Thus $A\nleq_\CRwr B$.
%
%$(ii) \Rightarrow (i)$. Suppose $(ii)$ holds and assume that $X \in \cs$ is not $A$-computably random. By Theorem~\ref{thm:sr-covers}, there exists an $A$-winning open set $\mathcal{U}$ such that all tails of~$X$ belong to~$\mathcal{U}$. By assumption, $\mathcal{U}$ can be covered by a $B$-winning open set~$\mathcal{V}$. Thus, all tails of $X$ belong to $\mathcal{V}$, which by Theorem~\ref{thm:cr-covers} implies that~$X$ is not $B$-computably random.  
%\end{proof}
%
%Note the similarity to the proof of Corollary~\ref{cor:low-cr-sr}.

\subsection{Lowness for randomness vs lowness for tests}\label{ssec:apps-tests}

When defining lowness for randomness notions, two approaches are possible. The obvious one is the one we have studied so far in this paper: $A \in \cs$ is low for a randomness notion~$\mathscr{R}$ if relativizing the notion~$\mathscr{R}$ to~$A$ leaves the set of random sequences unchanged. Now suppose that the notion~$\mathscr{R}$ is described via a family of tests (i.e., a sequence is random for the notion $\mathscr{R}$ if it passes all tests), like all of the notions we have presented above. A second possible lowness condition on~$A$ is to require that $A$-tests are not stronger than unrelativized tests, i.e., that for every $A$-test $\mathcal{T}$, there exists a test $\mathcal{T'}$ such that every sequence failing $\mathcal{T}$ also fails~$\mathcal{T'}$. While it is clear that (given a randomness notion defined by tests) lowness for tests implies lowness for randomness, the converse may not hold; a priori, it could be the case that many unrelativized tests are needed to cover a particular $A$-test. Nonetheless, there is currently no known example of a randomness notion (or a pair of randomness notions) for which lowness for tests is different from lowness for randomness. The results proven above provide a uniform explanation to why this is the case for lowness notions relating to Martin-L\"of randomness, computable randomness and Schnorr randomness. Let us prove for example that lowness for Schnorr randomness implies lowness for Schnorr tests (a result originally proven by Kjos-Hanssen et al.~\cite{Kjos-HanssenNS2005}). Let $A$ be low for Schnorr randomness, and consider an $A$-Schnorr test $(\mathcal{V}_n)_{n \in \N}$. Let us set $\mathcal{X} = \bigcap_n \mathcal{V}_n$. By Remark~\ref{rem:tests-and-covers} (relativized to $A$), there exists a bounded $A$-Schnorr open set $\mathcal{U}$ such that for any $X \in \mathcal{X}$, all tails of~$X$ are in $\mathcal{U}$. But since $A$ is low for Schnorr randomness, we now know that $\mathcal{U}$ must be covered by a bounded Schnorr open set $\mathcal{U}'$. This implies in particular that for any $X \in \mathcal{X}$, all tails of~$X$ are in $\mathcal{U}'$. Applying Remark~\ref{rem:tests-and-covers} again, there must exist a Schnorr test $(\mathcal{V}'_n)_{n \in \N}$ that covers~$\mathcal{X}$. In other words, $(\mathcal{V}'_n)_{n \in \N}$ covers $(\mathcal{V}_n)_{n \in \N}$. Therefore, $A$ is low for Schnorr tests. The same proof works for computable randomness and Martin-L\"of randomness, including lowness for pairs. For example, $A \in \Lowpair{\CR}{\SR}$ if and only if for every $A$-Schnorr test $(\mathcal{V}_n)_{n \in \N}$, there exists a single computable martingale~$d$ that succeeds against all $X \in \bigcap_n \mathcal{V}_n$.

\subsection{Partial relativization and lowness for weak $2$-randomness}\label{ssec:apps-W2R}

The results we presented above can be extended to the case of weak $2$-randomness by a \emph{partial relativization}. A ``partial relativization'' of a computability concept~$\mathscr{C}$ to an oracle $A$ consists in relativizing only some parts of the definition of~$\mathscr{C}$ to $A$. This device has already appeared in Section~\ref{ssec:apps-reducibilities}: $A\leq_\LR B$ can be seen as a partial relativization of ``$A$ is low for Martin-L\"of randomness'' to $B$. A full relativization would demand that every bounded $A\oplus B$-c.e.\ open set can be covered by a bounded $B$-c.e.\ open set. Another interesting example was given by Cole and Simpson \cite{CS07}, who define a notion of \emph{boundedly limit recursive in $X$} by partially relativizing the notion of $\omega$-c.e.; they relativize the approximation function but not the computable bound on the number of mind changes. More examples are given in Barmpalias et al.~\cite{BarmpaliasMN-sub}.

The following partial relativization will be central to this section.

\begin{defn}
Let $Z$ be a given oracle and let $(\mathcal{U}_e)_{e \in \N}$ be an effective enumeration of all c.e.\ open subsets of $\cs$. A $\plop{Z}$-Martin-L\"of test is a sequence $(\mathcal{U}_{f(n)})_{n \in \N}$ such that $f$ is computable in~$Z$ and $\mu(\mathcal{U}_{f(n)}) \leq 2^{-n}$ for all~$n$. A sequence $X \in \cs$ is $\plop{Z}$-Martin-L\"of random if it passes all $\plop{Z}$-Martin-L\"of tests. We denote by $\MLR\plop{Z}$ the set of $\plop{Z}$-Martin-L\"of random sequences.
\end{defn}

The concept of $\plop{Z}$-Martin-L\"of randomness is only a partial relativization of Martin-L\"of randomness: a full relativization to $Z$ would also allow, in the above definitions, the sets $\mathcal{U}_e$ to be $Z$-c.e. open, as opposed to just c.e.\ open. 

It turns out that Proposition~\ref{prop:closure} extends to partially relativized Martin-L\"of tests. 

\begin{prop} \label{prop:plop-closure}
For any $Z \in \cs$, the hypotheses (P1,P2,P3) of Lemma~\ref{lem:main} are satisfied when $\mathbf{C}$ is the class of bounded c.e.\ open sets and $(\mathcal{T}^e)_{e \in \N}$ the family of $\plop{Z}$-Martin-L\"of tests. 
\end{prop}

\begin{proof}
The proof is identical to the proof of Proposition~\ref{prop:closure}, as the latter does not use the uniform enumerability of Martin-L\"of tests, but solely the fact that $\mu(\mathcal{T}^{(e)}_n) \leq 2^{-n}$ for all $e,n$. 
\end{proof}

It follows from this proposition that, when it comes to lowness properties, partially relativized Martin-L\"of randomness behaves exactly like Martin-L\"of randomness (at least for the lowness properties discussed in this paper; it is not the case for some other notions. For example, the class $\Lowpair{\MLR}{\mathrm{KR}}$ is different from $\Lowpair{\SMLR}{\mathrm{KR}}$, where $\mathrm{KR}$ is the class of Kurtz random sequences). 

\begin{prop} \label{prop:plop-lowness}
For any~$Z$ in $\cs$, the following equalities hold.\\
(i) $\Lowpair{\MLR\plop{Z}}{\MLR} = \Low{\MLR}$\\
(ii) $\Lowpair{\MLR\plop{Z}}{\CR} = \Lowpair{\MLR}{\CR}$\\
(iii) $\Lowpair{\MLR\plop{Z}}{\SR} = \Lowpair{\MLR}{\SR}$
\end{prop}

\begin{proof}
The proofs of the three items are almost identical. Let us prove for example item (ii). Since $\MLR\plop{Z} \subseteq \MLR$, it is clear by definition that $\Lowpair{\MLR}{\CR} \subseteq \Lowpair{\MLR\plop{Z}}{\CR}$. Now take $A \notin \Lowpair{\MLR}{\CR}$. By the discussion of Section~\ref{ssec:apps-lowness}, this means that there exists an $A$-winning set $U$ such that $[U]$ is covered by no bounded c.e.\ open set. Therefore, one can apply Lemma~\ref{lem:main} with $\mathbf{C}$ the class of bounded c.e.\ open sets and $(\mathcal{T}^e)_{e \in \N}$ the family of $\plop{Z}$-Martin-L\"of tests (which is possible, by Proposition~\ref{prop:plop-closure}), to get an $X \in U^\omega$ that passes all $\plop{Z}$-Martin-L\"of tests. So $X$ is $\plop{Z}$-Martin-L\"of random, and, by Theorem~\ref{thm:cr-covers} (relativized to $A$), $X$ is not $A$-computably random. Thus $A \notin \Lowpair{\MLR\plop{Z}}{\CR}$. 
\end{proof}

Partial relativization of Martin-L\"of randomness is useful to study higher notions of randomness. For example, given an effective enumeration $(\mathcal{U}_e)_{e \in \N}$, the set $(e,k)$ of pairs such that $\mu(\mathcal{U}_e)< 2^{-k}$ is $\jump$-enumerable (because $\mu(\mathcal{U}_e)$ is $\jump$-computable, uniformly in~$e$). Thus, given a generalized Martin-L\"of test $(\mathcal{V}_n)_{n \in \N}$, one can $\jump$-compute a sequence $k_1 < k_2 < k_3 < \ldots$ such that $\mu(\mathcal{V}_{k_i}) < 2^{-i}$ for all~$i$. Therefore the generalized Martin-L\"of test  $(\mathcal{V}_n)_{n \in \N}$ is covered by the $\plop{\jump}$-Martin-L\"of test $\mu(\mathcal{V}_{k_i})_{i \in \N}$. This shows that $\plop{\jump}$-Martin-L\"of randomness implies weak $2$-randomness, and yields the following corollary. 

\begin{cor} \label{cor:w2r-lowness}
The following equalities hold.\\
(i) $\Lowpair{\SMLR}{\MLR} = \Low{\MLR}$\\
(ii) $\Lowpair{\SMLR}{\CR} = \Lowpair{\MLR}{\CR}$\\
(iii) $\Lowpair{\SMLR}{\SR} = \Lowpair{\MLR}{\SR}$
\end{cor}

\begin{proof}
Immediate from Proposition~\ref{prop:plop-lowness} (with $Z=\jump$) and the fact that $\MLR\plop{\jump} \subseteq \SMLR$. 
\end{proof}

Item (iii) answers a question of Nies~\cite{Nies2009} (see also Franklin~\cite{Franklin}, Greenberg and Miller~\cite{GreenbergM2009}). Although items (i) and (ii) were already known (proven respectively by Downey et al.~\cite{DowneyNWY2006} and Nies~\cite{Nies2009}), the proofs presented in this paper are simpler than the original ones. Note that Nies proved (ii) by showing that $\Lowpair{\SMLR}{\CR}=\Low{\MLR}$, which is more than we show above. However, in Theorem~\ref{thm:low-mlr-cr} we give a short proof that $\Lowpair{\MLR}{\CR}=\Low{\MLR}$, so together with Corollary~\ref{cor:w2r-lowness}(ii), we have reproved the stronger result.

\subsection{Highness notions}\label{ssec:apps-highness}

For a given pair $\mathscr{R}, \mathscr{S}$ of randomness notions with $\mathscr{R} \subseteq \mathscr{S}$, the class $\Lowpair{\mathscr{R}}{\mathscr{S}}$ denotes the set of oracles $A \in \cs$ that are weak enough to have $\mathscr{R} \subseteq \mathscr{S}^A$. It is natural to look at the dual concept of \emph{highness}, i.e.\ the set of oracles $A \in \cs$ that are powerful enough to have $\mathscr{S}^A \subseteq \mathscr{R}$, which we denote by $\Highpair{\mathscr{S}}{\mathscr{R}}$. These classes have primarily been studied by Franklin et al.~\cite{FranklinSY-sub} and Barmpalias et al.~\cite{BarmpaliasMN-sub}. Most highness notions involving the classes $\SMLR$, $\MLR$, $\CR$, and $\SR$ have been characterized: Franklin et al. showed that the classes $\Highpair{\SR}{\mathscr{R}}$ for $\mathscr{R}=\SMLR, \MLR, \CR$ are all equal to the set of $A$ such that $A \geq_T \jump$; Barmpalias et al.\ proved that the class $\Highpair{\MLR}{\SMLR}$ is equal to the set of sequences $A$ such that no $\jump$-computable function is diagonally non-computable relative to~$A$. Miller et al.~\cite{MNgRuTA} have recently given another characterization of $\Highpair{\MLR}{\SMLR}$; they proved that $A\notin\Highpair{\MLR}{\SMLR}$ iff every partial $A$-computable function is dominated by a $\jump$-computable function.

The classes $\Highpair{\CR}{\MLR}$ and $\Highpair{\CR}{\SMLR}$ on the other hand are not fully understood yet. Franklin et al.\ proved that every $A\in\Highpair{\CR}{\MLR}$ computes a Martin-L\"of random sequence. Kastermans, Lempp and Miller (unpublished) gave an alternative proof of this fact by showing that if $A\in\Highpair{\CR}{\MLR}$, then there is an $A$-computable martingale that succeeds against all non-Martin-L\"of random sequences. This result is an easy consequence of Lemma~\ref{lem:main}; we present the proof here.

\begin{prop}\label{prop:high-cr-mlr}
Let $A \in \cs$. The following are equivalent.\\
(i) $A \in \Highpair{\CR}{\MLR}$\\
(ii) Every bounded c.e.\ open set is covered by an $A$-winning set.\\
(iii) The first level\, $\mathcal{U}_1$ of a universal Martin-L\"of test is covered by an $A$-winning set.\\
(iv) There exists an $A$-computable martingale that succeeds against all $X$ that are not Martin-L\"of random.
\end{prop}
\begin{proof}
$(i) \Rightarrow (ii)$. If $(ii)$ does not hold, there is a bounded c.e.\ open set $\mathcal{U}$ that cannot be covered by any $A$-winning open set. Let~$U$ be a c.e.\ prefix-free set of strings generating~$\mathcal{U}$. We can apply Lemma~\ref{lem:main} with $\mathbf{C}$ the class of $A$-winning open sets and $\mathcal{T}$ the family tests induced by exactly $A$-computable martingales. To see that this is allowed, relativize Proposition~\ref{prop:closure} to $A$. Lemma~\ref{lem:main} gives us a sequence $X\in U^\omega$ that passes all tests induced by exactly $A$-computable martingales (hence is $A$-computably random). By Theorem~\ref{thm:mlr-covers}, $X$ is not Martin-L\"of random, so $A \notin \Highpair{\CR}{\MLR}$. 

$(ii) \Rightarrow (iii)$ is immediate, as is $(iv) \Rightarrow (i)$. This leaves $(iii) \Rightarrow (iv)$. Let $U$ be an $A$-winning set such that $[U]$ covers $\mathcal{U}_1$. Note that $U$ is prefix-free, by definition. The proof of $(iii) \Rightarrow (i)$ in Theorem~\ref{thm:cr-covers} (relativized to $A$) produces an $A$-computable martingale $d$ that succeeds against all $X\in U^\omega$. Now assume that $X$ is not Martin-L\"of random. Then all tails of $X$ are contained in $\mathcal{U}_1$, hence in $[U]$. This implies that $X\in U^\omega$, so $d$ succeeds on $X$. 
\end{proof}

To see that Proposition~\ref{prop:high-cr-mlr} implies the Franklin, Stephan and Yu result, note that every martingale computes a sequence on which it does not succeed. Hence, if $A \in \Highpair{\CR}{\MLR}$, then $A$ computes a Martin-L\"of random sequence. In fact:

\begin{cor}
If $A \in  \Highpair{\CR}{\MLR}$, then there is an~$A$-Turing functional $\Phi: \cs \rightarrow \cs$ that is total, one-to-one, and such that for all $X \in \cs$, $\Phi(X) \in \MLR$. 
\end{cor}

\begin{proof}
If $A \in  \Highpair{\CR}{\MLR}$, let (by Proposition~\ref{prop:high-cr-mlr})~$d$ be a normed $A$-computable martingale that succeeds on all~$X$ that are not Martin-L\"of random. We shall build an $A$-computable tree, i.e., an $A$-computable total function $T: \fs \rightarrow \fs$ such that if $\sigma'$ is a strict extension of $\sigma$, then $T(\sigma')$ is a strict extension of $T(\sigma)$ and if $\sigma$ and $\sigma'$ are incomparable then so are $T(\sigma)$ and $T(\sigma')$. We ensure that~$d$ succeeds on no infinite path in this tree by imposing the condition that for any $\sigma \in \fs$, setting $\tau=T(\sigma)$, we have $d(\tau') \leq 2 - 2^{-|\sigma|}$ for all prefixes $\tau'$ of $\tau$. Then, it follows immediately that for any infinite path $Y \in \cs$ in~$T$, we have $d(Y \uh n) \leq 2$ for all~$n$.

The construction of~$T$ is done by induction on the length~$k$ of $\sigma$. We first set $T(\emptystring)=\emptystring$. As $d(\emptystring)=1$ ($d$ is normed) this satisfies the requirement. Now, suppose we have defined $T(\sigma)$ for all~$\sigma$ of length~$k$ respecting the above requirement. Let $\sigma$ be of length~$k$ and set $\tau=T(\sigma)$.  We have by assumption $d(\tau) \leq 2-2^{-k}$. Hence, by the Ville-Kolmogorov inequality, the set of sequences $X \in \cs$ that extend $\tau$ and, for all $n$, satisfy $d(X\uh n) < 2-2^{-(k+1)}$ has positive measure. Since $d$ is $A$-computable, using oracle~$A$ we can find two incomparable extensions $\tau_0$ and $\tau_1$ of $\tau$ such that $d(\tau') \leq 2 - 2^{-(k+1)}$ for any prefix $\tau'$ of $\tau_0$ or $\tau_1$. We then set $T(\sigma \iota)=\tau_\iota$ for $\iota \in \{0,1\}$. This concludes the induction.

Finally, the $A$-functional $\Phi$ is defined in a straightforward manner by setting~$\Phi^\sigma=T(\sigma)$ for all~$\sigma \in \fs$.  
\end{proof}

%--------------------------------------------------------
\section{Reformulation in other contexts: converging series and machines}\label{sec:other-refumulations}
%--------------------------------------------------------

\subsection{Randomness via machines and Kolmogorov complexity}

Although we had no need for it so far in this paper, it is well-known that Kolmogorov complexity provides an alternative and elegant way to characterize randomness. In particular, Levin and Schnorr independently showed that a sequence~$X \in \cs$ is Martin-L\"of random if and only if $\K(X \uh n) \geq n - O(1)$, where~$\K$ denotes prefix-free Kolmogorov complexity.

In this last section, we discuss how our results relate to Kolmogorov complexity, and more precisely to prefix-free complexity. A~\emph{prefix-free machine} is a partial computable function $M: \fs \rightarrow \fs$ whose domain is prefix-free and~$\K_M$ is the Kolmogorov complexity associated to~$M$. As usual, we fix an optimal prefix-free machine~$M_{opt}$; we abbreviate~$\K_{M_{opt}}$ by $\K$ and call it \emph{prefix-free Kolmogorov complexity}. 

It is also well known that prefix-free Kolmogorov complexity is tightly related to convergent series. By the Kraft-Chaitin theorem, if $f: \N \rightarrow \R^{\geq 0}$ is summable (i.e., $\sum_n f(n) < +\infty$) and left-c.e., then $\K \leq - \log f +O(1)$. Levin's coding theorem proves that there exists a maximal left-c.e.\ summable function $F: \N \rightarrow \R^{\geq 0}$ (i.e., for any other such function~$f$, $f=O(F)$), and that we precisely have~$\K = - \log F +O(1)$.

Nies~\cite{Nies2005} proved that a sequence~$A$ is low for~$\K$ (i.e., $\K^A$, the prefix-free Kolmogorov complexity relativized to the oracle~$A$, is equal to the unrelativized version~$\K$, up to an additive constant) if and only if~$A$ is low for Martin-L\"of randomness. By Levin's coding theorem, this can be rephrased as follows. 

\begin{thm}~\label{thm:low-for-K-via-sums}
The following are equivalent:\\
(i) $A \in \Low{\MLR}$\\
(ii) For every $A$-left-c.e.\ summable function $f: \N \rightarrow \R^{\geq 0}$, there exists a left-c.e.\ summable function~$g$ such that $f \leq g$. 
\end{thm}

\begin{rem}\label{rem:low-for-K-via-sums2}
Condition (ii) can be replaced by the seemingly weaker condition on $A$: ``For every $A$-\emph{computable} summable function $f: \N \rightarrow \R^{\geq 0}$, there exists a left-c.e.\ summable function~$g$ such that $f \leq g$''. Indeed, if $f$ is an $A$-left-c.e.\ summable function, then defining $\tilde{f}(\langle i,t \rangle)$ to be the increase of $f(i)$ at stage~$t$, the function $\tilde{f}$ is computable and summable (its sum is the same as that of~$f$), and if $h$ is a left-c.e.\ summable function that dominates~$\tilde{f}$, then for all~$i$, $f(i)$ is dominated by $g(i)=\sum_{t} h(\langle i,t \rangle)$ and it is clear that~$g$ is left-.c.e.\ and summable (its sum is the same as that of~$h$).
\end{rem}

Using techniques similar to Nies', we can prove the analogous results for $\Low{\SR}$ and $\Lowpair{\MLR}{\SR}$.  

\begin{prop}\label{prop:series-sr}
The following are equivalent:\\
(i) $A \in \Low{\SR}$\\
(ii) For every $A$-computable function $f: \N \rightarrow \R^{\geq 0}$ whose sum is finite and computable (or $A$-computable), there exists a computable function~$g$ whose sum is finite and computable and such that $f \leq g$. 
\end{prop}

\begin{prop}\label{prop:series-mlr-sr}
The following are equivalent:\\
(i) $A \in \Lowpair{\MLR}{\SR}$\\
(ii) For every $A$-computable function $f: \N \rightarrow \R^{\geq 0}$ whose sum is finite and computable (or $A$-computable), there exists a left-c.e.\ summable function~$g$ such that $f \leq g$.  
\end{prop}

\begin{rem}
The two theorems can be equivalently stated with a computable sum or $A$-computable sum because of the following simple observation. Suppose $S=\sum_n f(n)$ is finite and $A$-computable. Let $N$ be an integer larger than $S$. Then the function $\hat{f}$ defined by $\hat{f}(0)=f(0)+N-S$ and $\hat{f}(n)=f(n)$ for $n>0$ is $A$-computable (as $S$ is) and $\sum_n \hat{f}(n)=N$. Of course, if we have a summable function~$g$ which dominates~$\hat{f}$, it dominates~$f$ as well.  
\end{rem}

\begin{proof}
We first prove Proposition~\ref{prop:series-mlr-sr} which is slightly easier and we will later see how to adjust the proof  to get Proposition~\ref{prop:series-sr}.\\

$(ii) \Rightarrow (i)$. Let $A \in \cs$ satisfying the hypotheses of $(ii)$. Let $(\mathcal{U}_n)_{n \in \N}$ be an $A$-Schnorr test. We can assume that every $\mathcal{U}_n$ is generated by an $A$-computable subset $U_n$ of $\cs$ (here we use the well-known fact that every c.e.\ open set $\mathcal{U}$ is generated by a \emph{computable} subset $U$ of $\fs$, and an index for $U$ can be computed from an index of $\mathcal{U}$). Then, define the function $f: \fs \rightarrow \R^{\geq 0}$ by $f(\sigma)=2^{-|\sigma|}n$ where $n$ is the largest integer such that $\sigma \in U_n$ and $f(\sigma)=0$ if $\sigma$ belongs to no set $U_n$. Note that $f$ is $A$-computable since $\sigma \in U_n$ implies $n \leq \sigma$. Moreover,
\[
\sum_\sigma f(\sigma) \leq \sum_n \sum_{\sigma \in U_n} 2^{-|\sigma|}n \leq \sum_n n\, \mu(\mathcal{U}_n) \leq \sum_n n2^{-n} < +\infty.
\]  

Let us check that the sum $\sum_\sigma f(\sigma)$ is $A$-computable. In the following we implicitly use the oracle~$A$. Let $k$ be an integer. To compute $\sum_\sigma f(\sigma)$ with precision $2^{-k}$, we can find for each~$n$ a finite set of strings $W_n \subseteq U_n$ such that $\mu([U_n] \setminus [W_n]) < 2^{-k}/kn$.  For each~$n$:
\[
\sum_{\sigma \in U_n \setminus W_n} f(\sigma) = n \cdot \mu([U_n] \setminus [W_n]) \leq 2^{-k}/k
\]
Thus, 
\[
\sum_{n \leq k} \; \sum_{\sigma \in U_n \setminus W_n} f(\sigma) \leq 2^{-k}
\]
and also
\[
\sum_{n > k} f(\sigma) = \sum_{n>k} n \cdot  \mu(\mathcal{U}_n) = O(2^{-k})
\]
Taking the two together, this shows that $\sum_{n \leq k} \sum_{\sigma \in W_n} f(\sigma)$ is an approximation of $\sum_\sigma f(\sigma)$ with precision $O(2^{-k})$. Therefore $\sum_\sigma f(\sigma)$ is $A$-computable. We can thus apply the hypothesis of $(ii)$ (where $\N$ and $\fs$ are identified) to get a summable left-c.e.\ function~$g:\fs \rightarrow \R^{\geq 0}$ such that $f \leq g$. Then define, for all~$n$,
\[
V_n=\{\sigma \; : \; g(\sigma) > 2^{-|\sigma|}(n/2)\}
\]
and set $\mathcal{V}_n=[V_n]$. It is clear that $\mathcal{V}_n$ is a c.e.\ open set, uniformly in~$n$, that covers $\mathcal{U}_n$. Let $S$ be the sum of~Ê$g$. We have:
\[
\mu(\mathcal{V}_n) \leq \sum_{\sigma \in V_n} 2^{-|\sigma|} < \sum_{\sigma \in V_n} \frac{2g(\sigma)}{n} \leq \frac{2S}{n}.
\]
% Finally, it is easy to see that $\mu(\mathcal{V}_n)$ is computable uniformly in~$n$, as the sum $S$ of $g$ is itself computable.
Thus $(\mathcal{V}_n)_{n \in \N}$ is a Martin-L\"of test (in the general sense of Remark~\ref{rem:test-note}) that covers $(\mathcal{V}_n)_{n \in \N}$. We have shown that every $A$-Schnorr test is covered by a Martin-L\"of test, so~$A \in \Lowpair{\MLR}{\SR}$.\\
 
$(i) \Rightarrow (ii)$. For this proof it is convenient to identify $\cs$ with the space $[0,1]^\N$ (using the usual identification of $\cs$ to $(\cs)^\N$ and then of $\cs$ to $[0,1]$; the non-uniqueness of binary expansion for dyadic rationals does not cause any problems here).

For all $n \in \N$ and $\alpha \in [0,1]$, set
\[
\mathcal{B}_{n, \alpha} = \{X \in [0,1]^\N \; : \; X_n \in [0, \alpha)\}.
\]
Now, let $f:\N \rightarrow \R^{\geq 0}$ be an $A$-computable function whose sum is finite and $A$-computable. Without loss of generality, we can assume that $f(n) \leq 1$ for all~$n$ (otherwise we divide $f$ by a constant $C$ to make this true, then find the desired function~$g$ that dominates~$f/C$, and observe that $C\cdot g$ dominates~$f$ and is as wanted). Consider the set
\[
\mathcal{U} = \bigcup_n \mathcal{B}_{n,f(n)}.
\]
$\mathcal{U}$ is a c.e.\ open set. Also, observe that if $n \not= m$, the sets $\mathcal{B}_{n, \alpha}$ and $\mathcal{B}_{m, \beta}$ correspond to independent events. Thus,
\[
\mu(\mathcal{U}) = 1 - \prod_{n} (1-\mu(\mathcal{B}_{n,f(n)})) = 1 - \prod_{n} (1-f(n)).
\]
This can be reformulated as
\[
\log(1-\mu(\mathcal{U})) = \sum_n \log(1-f(n)).
\]
Since $f(n)$ tends to~$0$, we have $\log(1-f(n)) \sim -f(n)$. This implies that $\sum_n \log(1-f(n))$ is finite and $A$-computable (as $\sum_n f(n)$ is). Thus, $\mathcal{U}$ is a bounded $A$-Schnorr open set. Since $A \in \Lowpair{\MLR}{\SR}$, we know that $\mathcal{U}$ must be covered by some bounded c.e.\ open set $\mathcal{V}$. Having such a set, we define
\[
g(n) = \sup \{\alpha \in [0,1] \; : \; \mathcal{B}_{n,\alpha} \subseteq \mathcal{V}\}.
\]

It is clear that $f \leq g$, and that~$g$ is left-c.e.\ as $\mathcal{V}$ is a c.e.\ open set. The sum $\sum_n g(n)$ is bounded because $\prod_n (1-g(n)) >0$, the latter being equal to $1-\mu(\bigcup_n \mathcal{B}_{n,g(n)}) \geq 1-\mu(\mathcal{V}) > 0$. 
\end{proof}

\begin{proof}[Proof of Proposition~\ref{prop:series-sr}.]
$(ii) \Rightarrow (i)$. Like for Proposition~\ref{prop:series-mlr-sr}, take an $A$-Schnorr test $\mathcal{U}_n$ and define the function~$f$ as before. By $(ii)$, there exists a computable function~$g:\fs \rightarrow \R^{\geq 0}$ whose sum $\sum_n g(n)$ is finite and computable and such that $f \leq g$. Up to replacing $g(n)$ by its approximation $g(n)[n]+2^{-n}$ (whose sum is still finite and computable), we can assume that~$g$ is exactly computable, i.e. computable as a function from $\N$ to $\Q^{\geq 0}$. Again  define, for all~$n$,
\[
V_n=\{\sigma \; : \; g(\sigma) > 2^{-|\sigma|}(n/2)\}
\]
and set $\mathcal{V}_n=[V_n]$. As before, $\mathcal{V}_n$ covers $\mathcal{U}_n$ and $\mu(\mathcal{V}_n) = O(1/n)$. Here, since we assumed $g$ to be exactly computable, the set $V_n$ is a computable set of strings. It remains to check that $\mu(\mathcal{V}_n)$ is computable uniformly in~$n$. Since $\sum_\sigma g(\sigma)$ is computable, for any given~$k$ one can effectively find $N=N(k)$ such that $\sum_{|\sigma|\geq N} g(\sigma) < 2^{-k}$.  Then 
\[
\sum_{\substack{|\sigma| \geq N\\ \sigma \in V_n} } 2^{-|\sigma|} \leq ( 2/n)\sum_{\substack{|\sigma|\geq N\\ \sigma \in V_n} } g(\sigma) < (2/n)\cdot 2^{-k}
\]
This shows that $\sum_{\sigma \in V_n} 2^{-|\sigma|}$ is computable uniformly in~$n$, and therefore so is $\mu(\mathcal{V}_n)$ (note that  $\sum_{\sigma \in V_n} 2^{-|\sigma|}$ and $\mu(\mathcal{V}_n)$ might be different but all that matters is that we can bound the tail sum). 

$(i) \Rightarrow (ii)$. Again we look at the open set 
\[
\mathcal{U} = \bigcup_n \mathcal{B}_{n,f(n)}.
\]
and by the hypothesis there exists a bounded Schnorr open~$\mathcal{V}$ which covers~$\mathcal{U}$. Let $\delta>0$ be such that $\mu(\mathcal{V}) < 1-\delta$. For all~$n$, let $\mathcal{V}[n]$ be the approximation of $\mathcal{V}$ with precision~$2^{-n}$. That is, $\mathcal{V}[n]$ is a clopen set for which an exact index can be uniformly computed in~$n$, and $\mu(\mathcal{V} \setminus \mathcal{V}[n]) < 2^{-n}$. Now define the function~$g$ by
\[
g(n) = \max \{ \alpha \in [0,1] \; : \; \mu(\mathcal{B}_{n,\alpha} \setminus \mathcal{V}[n]) \leq 2^{-n-c}\}.
\]
where~$c$ is a positive constant to be specified shortly. Note that $g$ is computable and $g \geq f$ because for all~$n$, $\mathcal{B}_{n,f(n)} \subseteq \mathcal{U} \subseteq \mathcal{V}$. The sum $\sum_n g(n)$ is bounded because $\mathcal{V}$ covers $\bigcup_n \mathcal{B}_{n,g(n)}$ up to measure $\sum_n 2^{-n-c}=2^{-c+1}$. Thus, for~$c$ large enough 
\[
\mu\left(\bigcup_n \mathcal{B}_{n,g(n)}\right) \leq 1-\delta+2^{-c+1} < 1
\]
and thus $1-\prod_n (1-g(n)) < 1$, which implies $\sum_n g(n) < \infty$. It remains to show that $\sum_n g(n)$ is a computable real, or equivalently, that given $\varepsilon$, one can effectively find an~$N$ such that $\sum_{n>N} g(n) \leq \varepsilon$. Let~$k$ be a fixed integer. The set $\mathcal{V}[k]$ is a clopen set, therefore for all but finitely many~$n$ such that the $\mathcal{B}_{n,g(n)}$, and moreover one can effectively find given $k$ an integer~$N=N(k)$, which we can assume to be greater than~$k$, such that $\mathcal{V}[k]$ is independent from the family of sets $\{\mathcal{B}_{n,g(n)} : n \geq N\}$. By this independence, we have:
\begin{eqnarray}
\mu\left( \bigcup_{n>N} \mathcal{B}_{n,g(n)} \setminus \mathcal{V}[k] \right) & =  & (1-\mu(\mathcal{V}[k])) \cdot \mu\left( \bigcup_{n>N} \mathcal{B}_{n,g(n)}\right) \\ & > & \delta \cdot \mu\left( \bigcup_{n>N} \mathcal{B}_{n,g(n)}\right)\label{eq:1}
\end{eqnarray}
On the other hand:
\begin{equation} \label{eq:2}
\mu\left( \bigcup_{n>N} \mathcal{B}_{n,g(n)} \setminus \mathcal{V} \right) \leq  \sum_{n>N} 2^{-n-c} = 2^{-N-c}
\end{equation} 
Let $\mathcal{W}=\mathcal{V} \setminus \mathcal{V}[k]$. By (\ref{eq:1}) and (\ref{eq:2}), we have

\begin{equation}
\mu\left( \bigcup_{n>N} \mathcal{B}_{n,g(n)} \setminus \mathcal{W} \right) \leq 2^{-N-c}+(1-\delta) \cdot \mu\left( \bigcup_{n>N} \mathcal{B}_{n,g(n)}\right)
\end{equation} 
and thus 
\begin{equation}
\mu(\mathcal{W}) +2^{-N-c} \geq \delta \cdot \mu\left( \bigcup_{n>N} \mathcal{B}_{n,g(n)}\right)
\end{equation} 
But $\mu(\mathcal{W}) < 2^{-k}$ and $2^{-N-c} < 2^{-k}$. Thus 
\begin{equation}
1 - \prod_{n>N} (1-g(n)) = \mu\left( \bigcup_{n>N} \mathcal{B}_{n,g(n)}\right) < 2^{-k+1}/\delta
\end{equation} 
or equivalently
\begin{equation}
\prod_{n>N} (1-g(n)) > 1- 2^{-k+1}/\delta
\end{equation} 
Composing with $-\log$ on both sides, we get 
\begin{equation}
\sum_{n>N(k)} g(n) < -\log(1- 2^{-k+1}/\delta) = 2^{-k+o(k)}
\end{equation} 
This last equation allows us to effectively compute for all~$k$ an approximation of $\sum_{n} g(n)$ (namely: $\sum_{n\leq N(k)} g(n)$), hence $\sum_{n} g(n)$ is computable. 

\end{proof}

Downey and Griffiths~\cite{DowneyG2004} gave a Levin-Schnorr-like characterization of Schnorr randomness by restricting Kolmogorov complexity to a specific class of prefix-free machines. They proved that a sequence~$X \in \cs$ is Schnorr random if and only if for every computable measure machine~$M$, one has $\K_M(X \uh n) \geq n - O(1)$, where a \emph{computable measure machine} is a prefix-free machine whose domain has computable measure (i.e., is a Schnorr set). Furthermore, Downey et al.~\cite{DowneyGMN2008} showed that an analogue of Nies' result ``low for random equals low for $\K$'' holds for Schnorr randomness, as explained in the next proposition. 

\begin{prop}
The following are equivalent:\\
(i) $A \in \Low{\SR}$\\
(ii) For any~$A$-computable measure machine~$M$, there exists a computable measure machine~$M'$ such that $\K_{M'} \leq \K_M+O(1)$.
\end{prop}

It turns out that the results we proved earlier in this section allow us to give a short proof of the above. 

\begin{proof}
The part~$(ii) \Rightarrow (i)$ is clear from the Downey-Griffith characterization of Schnorr randomness. For the reverse direction, let $M$ be an $A$-computable measure machine. Let~$f: \fs \rightarrow \R^{\geq 0}$ be the function defined by
\[
f(\sigma)=2^{-\K_M(\sigma)}
\]
It is easy to see that~$f$ is an $A$-left-c.e.\ function. Moreover, if we enumerate the domain of~$M$, whenever a new~$p$ is found (i.e., at that stage the measure of $\dom(M)$ is increased by $2^{-|p|}$, this increases the sum $\sum_\sigma f(\sigma)$ by either $0$ or~$2^{-|p|}$ (depending on whether we had already enumerated some $p'$ with $|p'| \leq |p|$ and $M(p)=M(p')$ or not before that stage). This shows that the sum $\sum_\sigma f(\sigma)$ can be computed from the measure of~$\dom(M)$, which is $A$-computable. Hence, $\sum_\sigma f(\sigma)$ is $A$-computable. We can therefore apply Proposition~\ref{prop:series-sr} to get a left-c.e.\ summable function $g: \fs \rightarrow \R^{\geq 0}$ whose sum is computable and such that $f \leq g$. Let~$c \in \N$ be a constant such that $\sum_\sigma f(\sigma) \leq 2^c$. We enumerate the Kraft-Chaitin set
\[
L= \big\{ (k, \sigma)\; : \; g(\sigma) \geq 2^{-k+c+1}\big\}.
\]
We have
\[
\sum_{(k, \sigma) \in L} 2^{-k} = \sum_\sigma 2^{- \lceil \log g(\sigma) - c - 1\rceil +1} \leq 2^{-c} \sum_\sigma g(\sigma) \leq 1,
\]
so~$L$ is indeed a Kraft-Chaitin set. Now, apply the Kraft-Chatin theorem to construct a machine~$M'$ whose domain is a prefix-free set $\{ p_{k,\sigma} \; : \; (k,\sigma) \in L\}$ with $|p_{k, \sigma}|=k$ and $M'(p_{k, \sigma})=\sigma$. It follows by construction that $M'$ is a computable measure machine (the measure of its domain is $\sum_\sigma 2^{- \lceil \log g(\sigma) - c - 1\rceil +1}$, which is computable because $\sum_\sigma g(\sigma)$ is computable) and
\[
\K_{M'} \leq - \log g + c +1 \leq - \log f + c +1 \leq \K_M + c + 1.
\] 
This completes the proof.
\end{proof}

In the same way, we can get the analogous result for the pair $(\MLR,\SR)$.

\begin{prop}
The following are equivalent:\\
(i) $A \in \Lowpair{\MLR}{\SR}$\\
(ii) For any~$A$-computable measure machine~$M$, we have $\K \leq \K_M + O(1)$.
\end{prop}

The proof is almost identical to the proof of the previous result (using Proposition~\ref{prop:series-mlr-sr} instead of Proposition~\ref{prop:series-sr}) and is left to the reader. \\

\subsection{A final application: $\Low{\MLR}=\Lowpair{\MLR}{\CR}$}\label{ssec:apps-winning}

We finish with an alternative proof that $\Low{\MLR}=\Lowpair{\MLR}{\CR}$. This was shown by Nies~\cite{Nies2005,Nies2009}, but the only known proof is long and technical. We believe that our proof is more comprehensible. It uses the covering characterization of $\Lowpair{\MLR}{\CR}$ and the characterization of $\Low{\MLR}$ via summable series (which as we pointed out, is an easy consequence of the coding theorem). Together with Corollary~\ref{cor:w2r-lowness}(ii), we in fact have an alternate proof that $\Low{\MLR}=\Lowpair{\SMLR}{\CR}$ (Nies~\cite{Nies2009}).

\begin{thm} \label{thm:low-mlr-cr}
The classes~$\Low{\MLR}$ and $\Lowpair{\MLR}{\CR}$ coincide.
\end{thm} 

\begin{proof}
It is clear that $\Low{\MLR}$ is contained in $\Lowpair{\MLR}{\CR}$. We need to show the reverse implication. Let $A \in \Lowpair{\MLR}{\CR}$. To show that $A \in \Low{\MLR}$, we use Theorem~\ref{thm:low-for-K-via-sums} and Remark~\ref{rem:low-for-K-via-sums2}: we consider an $A$-computable summable function $f: \N \rightarrow \R^{\geq 0}$ and we will show that there exists a left-c.e.\ summable function~$g$ such that $f \leq g$. Without loss of generality, we assume that $\sum_i f(i) \leq 1$ and that the $f(i)$ are dyadic rational numbers in virtue of which we set $f(i)=2^{-a_i}$, the sequence of $a_i$ being $A$-computable. The proof's strategy is the following: we shall ``encode'' the function~$f$ in an~$A$-winning open set $V$. Then, since $A \in \Lowpair{\MLR}{\CR}$, there exists a bounded c.e.\ open set $\mathcal{W}$ that covers~$\mathcal{V}$, and from $\mathcal{V}$ we will build a left-c.e.\ summable function~$g$ that dominates~$f$. 

To construct the set~$\mathcal{U}$, we first consider a computable partition of $\N$ into intervals $I_{i,l}$ where for all~$i,l$, the length of $I_{i,l}$ is $l$ (the order in which the intervals are placed does not matter). For all~Ê$i,l$, we consider the set $\mathcal{Z}_{i,l}$ of sequences~$X$ such that $X(n)=0$ for all~$n \in I_{i,l}$. Note that $\mu(\mathcal{Z}_{i,l})=2^{-l}$ and the $\mathcal{Z}_{i,l}$ are pairwise independent. Then, we set
\[
\mathcal{U} = \bigcup_{i} \mathcal{Z}_{i,a_i}
\]
Since the $\mathcal{Z}_{i,a_i}$ are all independent, the measure of $\mathcal{U}$ is $1-\prod_i (1-2^{-a_i})$, which is smaller than~$1$ as the sum~$\sum_i 2^{-a_i}$ is finite. We now show that some $A$-computable martingale wins money against every member of $\mathcal{U}$. Let $q>1$ be a rational number such that $\sum_i 2^{-a_i} < 1/q$. We define an $A$-computable martingale~$d$ as follows. For each~$i$, we reserve an amount $q2^{-a_i}$. When betting on a sequence~$X$, at the start of an interval $I_{i,l}$, the martingale $d$ check whether $l=a_i$. If not, $d$ does not bet on any position of the interval. If so, then $d$ uses the reserved capital $q2^{-a_i}$ to bet that $X$ contains only zeroes on the interval $I_{i,a_i}$. This is done by first betting an amount $x=q2^{-a_i}$ on~$0$, then if correct an amount $2x$, then $4x$, etc., stopping if some guess was incorrect. If all of its guesses are correct, the capital of~$d$ has increased by $x2^{a_i}=q$ at the end of the interval $I_{i,a_i}$. Thus, the martingale $d$ is $A$-computable and reaches a capital of at least~$q$ on every element~$X$ of $\mathcal{U}$. Hence $\mathcal{U}$ is contained in the winning set $\mathcal{V}=[V]$ with 
\[
V= \{\sigma~ \text{minimal s.t.\ } d(\sigma) \geq q\}
\] 
Now, since $A \in \Lowpair{\MLR}{\CR}$, there exists an unrelativized bounded c.e.\ open set $\mathcal{W}$ that covers~$\mathcal{V}$. Then, set for all~$i$, $g(i)=2^{-b_i}$ with
\[
b_i = \min \{l\;  \mid \; \mathcal{Z}_{i,l} \subseteq \mathcal{W}\}
\]
The sequence~$b_i$ is right-c.e.\, and by construction~$b_i \leq a_i$ as $\mathcal{Z}_{i,a_i} \subseteq \mathcal{U} \subseteq \mathcal{W}$, so $g(i) \geq f(i)$. Finally, notice that the sum $\sum_i 2^{-b_i}$ is finite as
\[
1 > \mu(\mathcal{W}) \geq 1- \prod_i (1- 2^{-b_i})
\]
hence $\prod_i (1- 2^{-b_i})>0$ so $\sum_i 2^{-b_i}$ converges. Therefore the function~$g$ is as wanted. 
\end{proof}

%\section{Conclusion}

%
\bibliography{open_coverings}
\bibliographystyle{plain}

\end{document}